\RequirePackage{rotating}
\documentclass[12pt]{amsart}
\usepackage[margin=1.2in,footskip=0.25in]{geometry}

\usepackage[graphicx]{realboxes}
\usepackage{float}
\restylefloat{table}
\usepackage{placeins}
\usepackage{amsmath}

\usepackage{amssymb}
\usepackage{epsfig}
\usepackage{wasysym}
\usepackage{graphicx}
\usepackage{tikz}
\usepackage{tikz-cd}
\usepackage{bm}
\usepackage{xcolor}
\usepackage{listings}
\numberwithin{equation}{section}
\usepackage{extpfeil}
\usepackage{array}
\usepackage{adjustbox}
\usepackage{longtable}
\usepackage{makecell}
\usepackage[all]{xy}
\usepackage{longtable}
\usepackage{rotating}
\input xy
\xyoption{all}
\usepackage{multirow}

\usetikzlibrary{positioning}
\usepackage[colorlinks=false,urlbordercolor=white]{hyperref}
  
\tikzset{sgplattice/.style={inner sep=1pt,norm/.style={red!50!blue},char/.style={blue!50!black},
  lin/.style={black!50}},cnj/.style={black!50,yshift=-2.5pt,left=-1pt of #1,scale=0.5,fill=white}}

\DeclareFontFamily{U}{mathb}{\hyphenchar\font45}
\DeclareFontShape{U}{mathb}{m}{n}{
      <5> <6> <7> <8> <9> <10> gen * mathb
      <10.95> mathb10 <12> <14.4> <17.28> <20.74> <24.88> mathb12
      }{}
\DeclareSymbolFont{mathb}{U}{mathb}{m}{n}
\DeclareMathSymbol{\righttoleftarrow}{3}{mathb}{"FD}

\calclayout
\allowdisplaybreaks[3]

\theoremstyle{plain}
\newtheorem{prop}{Proposition}[section]

\newtheorem{theo}[prop]{Theorem}
\newtheorem{coro}[prop]{Corollary}

\newtheorem{lemm}[prop]{Lemma}

\theoremstyle{definition}
\newtheorem{defi}[prop]{Definition}

\newtheorem{ques}[prop]{Question}
\newtheorem{conj}[prop]{Conjecture}

\newtheorem{rema}[prop]{Remark}

\newtheorem{exam}[prop]{Example}

\def\cN{{\mathcal N}}
\def\cO{{\mathcal O}}

\def\sA{{\mathsf A}}

\def\bc{{\mathbf c}}

\def\dd{{\mathrm d}}

\def\bG{{\mathbb G}}
\def\bP{{\mathbb P}}
\def\bQ{{\mathbb Q}}

\def\bZ{{\mathbb Z}}

\def\bc{{\mathbf c}}

\def\bN{{\mathbb N}}
\def\bC{{\mathbb C}}

\def\rH{{\mathrm H}}

\def\bF{{\mathbb F}}

\def\Bl{\mathrm{Bl}}
\def\Pic{\mathrm{Pic}}

\def\Gal{\mathrm{Gal}}

\def\Gr{\mathrm{Gr}}

\def\SL{\mathsf{SL}}

\def\PGL{\mathsf{PGL}}

\def\Burn{\mathrm{Burn}}
\def\Burnf{\mathbf{Burn}}
\def\dBurn{\mathbf{DivBurn}}
\def\Bir{\mathrm{Bir}}

\def\lim{\mathrm{lim}}

\def\Cr{\mathrm{Cr}}

\def\Sing{\mathrm{Sing}}

\makeatother
\makeatletter

\begin{document}
\title[]{Birational invariants of volume preserving maps}

\author[K. Loginov]{Konstantin Loginov}
\address{
Steklov Mathematical Institute of Russian Academy of Sciences, Moscow, Russia
}

\email{loginov@mi-ras.ru}

\author[Zh. Zhang]{Zhijia Zhang}

\address{
Courant Institute,
  251 Mercer Street,
  New York, NY 10012, USA
}

\email{zhijia.zhang@cims.nyu.edu}

\date{\today}

\begin{abstract}
We study the group of birational automorphisms of the $n$-dimensional projective space that preserve the standard torus invariant  volume form with logarithmic poles.  
We prove that this group is not generated by pseudo-regularizable maps for $n\geq 4$ over $\mathbb{C}$, and for $n\geq 3$ over number fields. 
As a corollary, we show that this group is not simple in these cases.
\end{abstract}

\maketitle
\tableofcontents

\section{Introduction}
Let $\Bbbk$ be a field of characteristic zero. It is a long-standing problem in birational geometry to understand the Cremona group $\Cr_n(\Bbbk)=\Bir_\Bbbk(\bP^n)$, the group of birational automorphisms over $\Bbbk$ of the $n$-dimensional projective space. 
The classical problem of finding explicit generating sets of $\Cr_n(\Bbbk)$ dates back to the 19th century. In dimension~$2$, the solution is known as the Noether-Castelnuovo theorem:
\begin{theo}[\cite{Noether, Castelnuovo}]
Over an algebraically closed field $\Bbbk$, the group $\mathrm{Cr}_2(\Bbbk)$ is generated by $\PGL_3(\Bbbk)=\mathrm{Aut}_{\Bbbk}(\bP^2)$, and the Cremona involution $\sigma$ acting on $\bP^2$ via
\[
\sigma: (x:y:z)\mapsto \left(\frac1x:\frac1y:\frac1z\right).
\]
\end{theo}
\noindent The situation is more complicated when $\Bbbk$ is not algebraically closed: the base loci of the maps generated by $\PGL_3(\Bbbk)$ and $\iota$ only consist of $\Bbbk$-rational points, so these maps form a proper subgroup of $\Cr_2(\Bbbk)$, see \cite{BHcre}. 
Generators of $\Cr_2(\Bbbk)$ over a nonclosed field $\Bbbk$ are described in \cite{Isk91}. 
Despite significant developments in modern birational geometry, the problem of finding explicit generating sets of $\Cr_n(\Bbbk)$ for $n\geq 3$ is still open. 

In many cases, it is also interesting to study birational maps preserving additional structures, e.g., symmetries or volume forms. In this paper, we study birational maps that preserve logarithmic volume forms (see Definition \ref{defi:logform}). In particular, consider the \emph{standard torus invariant volume form} on $\bP^n$, given by the formula
\begin{equation}
\label{eq-standard-toric-form}
\omega_n = \frac{\dd x_1}{x_1}\wedge\cdots\wedge\frac{\dd x_n}{x_n}
\end{equation}
in some affine chart $\mathbb{A}^n\subset \bP^n$. Let $\Bir_\Bbbk(\bP^n,\omega_n)$ be the group of {\em volume preserving} birational automorphisms over $\Bbbk$ of the pair $[\bP^n,\omega_n]$, i.e., birational automorphisms over $\Bbbk$  of $\bP^n$ that preserve $\omega_n$ (see Definition \ref{defi:vol}). Corti and Kaloghiros showed that over $\Bbbk=\bC$, this group is generated by volume preserving Sarkisov links \cite{Cortivolume}. We give a negative answer to the existence of a {\em simple} generating set of $\Bir_\Bbbk(\bP^n,\omega_n)$, for various fields $\Bbbk$ when $n\geq 3.$ 

In dimension $2$, an explicit set of generators of $\Bir_\bC(\bP^2,\omega_2)$ was found by Usnich \cite{Usn} and Blanc \cite{Bl}. Note that the standard torus action of $\bG_m^2$ 
on~$\bP^2$ preserves $\omega_2$. There is also an embedding of the group $\SL_2(\mathbb{Z})$ into $\Bir_\Bbbk(\bP^2,\omega_2)$, via automorphisms of $\mathbb G_m^2$ given by the following formula: 
\[
(x, y) \mapsto (x^a y^b, x^c y^d),
\]
for an element ${\tiny\begin{pmatrix}
    a&b\\
    c&d
\end{pmatrix}}$$\in\SL_2(\mathbb{Z})$ .

\begin{theo}[{\cite{Bl}}]\label{theo:blanc}
Let $\omega_2$ be the standard torus invariant volume form \eqref{eq-standard-toric-form} on~$\bP^2$ over an algebraically closed field $\Bbbk$.
Then the group $\Bir_\Bbbk(\bP^2, \omega_2)$ is generated by the subgroups $\bG_m^2$, $\SL_2(\mathbb{Z})$ described above, and the cluster transformation $\tau$ of order $5$: 
\[
\tau\colon (x,y)\mapsto \left(y, \frac{y+1}{x}\right).
\]
\end{theo}

A birational automorphism of $\bP^n$ is called {\em regularizable} if it is conjugate  in $\Cr_n(\Bbbk)$ to a {\em regular} automorphism of a rational variety. For example, the map $\tau$ in Theorem \ref{theo:blanc} is regularized on a del Pezzo surface of degree $5$. 
In particular, elements of finite order in $\Cr_n(\Bbbk)$ are regularizable. 
A birational automorphism is called \emph{pseudo-regularizable} if it is conjugate in $\Cr_n(\Bbbk)$ to a pseudo-automorphism of a rational variety, i.e., to a birational automorphism which is an isomorphism in codimension $1$ (see Example~\ref{exam:vanishpseudo}). 
Since $\SL_2(\mathbb{Z})$ can be generated by elements of finite order, it follows that both $\mathrm{Cr}_2(\Bbbk)$ and $\Bir_\Bbbk(\bP^2, \omega_2)$ are generated by regularizable elements if $\Bbbk$ is algebraically closed. 

When $n\geq 3$ and $\Bbbk$ is algebraically closed, it is known that $\Cr_n(\Bbbk)$ cannot be generated by regular automorphisms together with birational maps of bounded degree or countably many elements when $\Bbbk$ is uncountable \cite{IvanPan}.
 Recently, Lin and Shinder showed that $\Cr_n(\Bbbk)$ cannot be generated by pseudo-regularizable elements.

 \begin{theo}[{\cite{SL}}]\label{theo:sl}
     In each of the following cases, $\Cr_n(\Bbbk)$ is not generated by pseudo-regularizable elements:
     \begin{enumerate}
         \item $n\geq 3$ and $\Bbbk$ is a number field; or the function field of an algebraic variety over a number field, over a finite field, or over an algebraically closed field,
        \item $n\geq 4$ and $\Bbbk$ is a subfield of $\bC$,
        \item $n\geq 5$ and $\Bbbk$ is any infinite field.
     \end{enumerate}
 \end{theo}
 
Theorem~\ref{theo:sl} implies that any generating set of $\Cr_n(\Bbbk)$ in the given cases must be considerably intricate. Our main theorem is a generalization of this result to $\Bir_\Bbbk(\bP^n,\omega_n)$.

\begin{theo}[{Corollary \ref{coro:main3}, Corollary \ref{coro:main4}}]
\label{thm-main-thm}
Let $\omega_n$ be the standard torus invariant volume form \eqref{eq-standard-toric-form} on $\bP^n$. 
In each of the following cases, $\Bir_\Bbbk(\bP^n,\omega_n)$ is not generated by pseudo-regularizable elements:
\begin{enumerate}
    \item $n\geq3$ and $\Bbbk$ is a number field; or the function field of an algebraic variety over a number field, over a finite field, or over an algebraically closed field,
    \item $n\geq 4$ and $\Bbbk=\bC$.
\end{enumerate}
In particular, $\Bir_\Bbbk(\bP^n,\omega_n)$ in cases (1) or (2) is not generated by the standard torus action of $\bG_m^n$ on $\bP^n$ together with elements of finite order.


\end{theo}


Our main tool is a birational invariant $\mathbf{c}$ defined in \cite{CLKT} for volume preserving birational maps, in parallel with the invariant $c$ for birational maps introduced in \cite{SL}, see Section \ref{subsec-birational-invariants}. 
We show that the birational automorphisms used in \cite{SL} can be realized as volume preserving birational automorphisms of the pair $[\bP^n,\omega_n]$ with non-trivial invariant $\mathbf{c}$. This proves Theorem \ref{thm-main-thm} and leads to the following. 
\begin{theo}[{Corollary \ref{coro:nonsimple3}, Corollary \ref{coro:nonsimple4}}]
The group $\Bir_\Bbbk(\bP^n, \omega_n)$ is not simple when $n$ and $\Bbbk$ are as in cases (1) or (2) in Theorem \ref{thm-main-thm}.  
\end{theo}

Another motivation to study varieties equipped with volume forms comes from the theory of {\em log Calabi-Yau pairs}. A  log Calabi-Yau pair is a pair $(X,D)$ consisting of a normal proper variety~$X$ and an effective divisor $D$ with integral coefficients such that $K_X+D\sim 0$, see section \ref{sec-CY-coreg}. For every such pair, there exists a unique (up to scaling) volume form $\omega_D$ on $X$ satisfying $D+\mathrm{div}(\omega_D)=0$. It turns out that  studying volume preserving birational maps is crucial for the understanding of birational geometry of log Calabi-Yau pairs.

In a recent work \cite{araujo2023birational}, the group $\Bir_{\bC}(\bP^3,D)$ of birational automorphisms of $\bP^3$ preserving (up to scaling) the volume form $\omega_D$ associated with an irreducible quartic surface $D$ was studied. The pair $(\bP^3,D)$ in this case has coregularity $2$, see Definition \ref{sec-CY-coreg}. It was noted that the pair is birationally rigid when $D$ is smooth and general (see \cite[Theorem A]{araujo2023birational}) and the appearance of  singularities on $D$ enriches the birational geometry of the pair $(\bP^3,D)$. 

Our work concerns the opposite case: elements in $\Bir_\Bbbk(\bP^n,\omega_n)$ correspond to birational automorphisms of the log Calabi-Yau pairs $(\bP^n,D)$  with coregularity~$0$,
where $D$ is the union of $n+1$ coordinate hyperplanes. Correspondingly, we observe a different behaviour and a richer group structure of $\Bir_{\Bbbk}(\bP^n,\omega_n)$. It turns out that many well-known rational varieties admit a boundary of coregularity $0$, which makes them crepant equivalent (see Definition \ref{diag-crepant-equivalence}) to $\bP^n$ with the union of $n+1$ coordinate hyperplanes, see \cite{LMV24}.

Our constructions in Section~\ref{sect:ellip} and~\ref{sect:K3} imply that certain birational automorphisms of $\bP^n$ preserve an anti-canonical divisor $D$ of coregularity $0$. On the other hand, there exist birational automorphisms of $\bP^n$ which do not preserve any anti-canonical divisor of coregularity~$0$. The following question is interesting:
\begin{ques}
How to characterize elements in $ \mathrm{Cr}_n(\Bbbk)$ which preserve an anti-canonical divisor in $\bP^n$ over $\Bbbk$ of coregularity $0$?
\end{ques}

The group structure of $\Bir_{\Bbbk}(\bP^2, \omega_2)$ has also been studied in different contexts. A surjective (and non-injective) homomorphism from $\Bir_{\bC}(\bP^2, \omega_2)$ to the Thompson group, the group of piecewise-linear automorphisms of $\bZ^2$, was constructed in \cite{Usn}. It follows that $\Bir_\Bbbk(\bP^2, \omega_2)$ is not simple when $\Bbbk$ is algebraically closed. 
The construction of \cite{Usn} is clarified through the action of the group $\Bir_{\Bbbk}(\bP^2, \omega_2)$ on the space of valuations of the function field $\Bbbk(\bP^2)$, see  \cite{Favre}. 
From this perspective, the construction can be naturally generalized to higher dimensions. 
This leads to the following question.
\begin{ques}
Is the natural map from $\Bir_{\Bbbk}(\bP^n, \omega_n)$ to the group of piecewise-linear automorphisms of $\bZ^n$ always surjective?
\end{ques}

Here is a roadmap of this paper. In Section~\ref{sect:pre}, we recall basic notions from birational geometry and the Chambert-Loir--Kontsevich--Tschinkel Burnside formalism of varieties endowed with logarithmic volume forms. In Section~\ref{subsec-birational-invariants}, we study our main tool: the invariant $\mathbf c$ of volume preserving birational maps. Section~\ref{sect:ellip} proves the first assertion of Theorem~\ref{thm-main-thm}, using a classical birational automorphism of $\bP^3$. Section~\ref{sect:K3} proves the second assertion of Theorem~\ref{thm-main-thm}, using the Hassett--Lai birational automorphism of $\bP^4$ constructed in \cite{HL}.   

\

\noindent
{\bf Acknowledgments.} 
The first author is supported by the Russian Science Foundation under grant 24-71-10092. 
The authors are thankful to Yuri Tschinkel for his interest in this work and useful remarks. The authors are grateful to Alexander Kuznetsov for numerous suggestions on the exposition, and to Brendan Hassett and Constantin Shramov for helpful conversations.

\section{Preliminaries}\label{sect:pre}
Throughout, $\Bbbk$ is a field of characteristic $0$, not necessarily algebraically closed. 
We
use the language of the minimal model program, see
e.g. \cite{KM98}. 
\subsection{Pairs and singularities}
A 
morphism 
$
f\colon X \to Y
$
of normal proper varieties 
is called a {\em contraction} if
$
f_*\cO_X =
\cO_Y.
$
Note that a contraction is surjective and has connected fibers. A contraction $f$ is called a
\emph{fibration} if
$
\dim Y<\dim X.
$

A \emph{pair} (resp., a \emph{sub-pair}) $(X, D)$ consists
of a normal proper variety $X$ and a Weil $\mathbb{Q}$-divisor $D$ with
coefficients in $[0, 1]$ (resp., in $(-\infty, 1]$) such that $K_X + D$ is $\mathbb{Q}$-Cartier.
In this situation, we call $D$ a \emph{boundary} (resp., a \emph{sub-boundary}).
Recall that a divisor $D$ on a smooth variety has \emph{simple normal crossings} (snc for short), if all of its components are smooth, and any point in $D$ has an open neighborhood in the analytic topology that is analytically equivalent to the union of coordinate hyperplanes.

We say that $f\colon Y \to X$ is a
\emph{log resolution} of a sub-pair $(X,D)$ if $Y$ is smooth, and $\mathrm{Exc}(f)\cup \mathrm{supp}(f^{-1}_*D)$ is an snc divisor. 
Let $D_Y = -K_Y+f^*(K_X +D)$. 
Then $(Y, D_Y)$ is called the log pullback of $(X, D)$. If $(X, D)$ is a sub-pair and $f$ is a log resolution, then $(Y, D_Y)$ is called an \emph{snc modification}. 
The \emph{log discrepancy} $a(E, X, D)$ of a prime divisor $E$ on $Y$ with respect
to $(X, D)$ is defined as
$$
a(E, X, D):=1 - \mathrm{coeff}_E D_Y.
$$ We
say $(X, D)$ is lc (resp., klt) if $a(E, X, D)\geq 0$ (resp.,~$> 0$) for every such $E$ and for any log resolution $f$. We say that the pair is {plt}, if $a(E, X, D)>0$ holds for any $f$-exceptional divisor $E$ and for any log resolution $f$. 
We say that the pair is {dlt}, if $a(E, X, D)>0$ holds for any $f$-exceptional divisor $E$ and for some log resolution $f$.

An \emph{lc-place} of $(X, D)$ is a prime divisor $E$ on a birational model of~$X$, such that $a(E, X, D) = 0$. An
\emph{lc-center} is the image on $X$ of an lc-place.

\subsection{Dual complex}
\label{subsec-dual-complex}
Let $D=\sum D_i$ be a Cartier divisor on a smooth variety $X$. 
\emph{The dual complex}, denoted by $\mathcal{D}(D)$, of a simple normal crossing divisor $D=\sum_{i=1}^{r} D_i$ on a smooth variety $X$ is a CW-complex constructed as follows. 
The simplices $v_Z$ of $\mathcal{D}(D)$ are in bijection with irreducible components $Z$ of the intersection $\bigcap_{i\in I} D_i$ for any non-empty subset $I\subset \{ 1, \ldots, r\}$, and the vertices of $v_Z$ correspond to the components $D_i$ with $i\in I$. 
In particular, the dimension of $v_Z$ is equal to $\#I-1$. We call $Z$ a \emph{stratum} of $D$.


The gluing maps are constructed as follows. 
For any non-empty subset $I\subset \{ 1, \ldots, r\}$, let $Z\subset \bigcap_{i\in I} D_i$ be a stratum, and for any $j\in I$, let $W$ be the unique component of $\bigcap_{i\in I\setminus\{j\}} D_i$ containing $Z$. Then the gluing map is the inclusion of $v_W$ into $v_Z$ as a face of $v_Z$ that does not contain the vertex $v_i$ corresponding to $D_i$. Note that the dimension of $\mathcal{D}(D)$ does not exceed $\dim X-1$. If $\mathcal{D}(D)$ is empty, i.e., $D=0$, we say that $\dim \mathcal{D}(D)=-1$. 

 We denote by $D^{=1}$  the sum of the components of $D$ with coefficient $1$. For an lc pair $(X, D)$, we define $\mathcal{D}(X, D)$ as $\mathcal{D}(D_Y^{=1})$ where $f\colon (Y, D_Y)\to (X, D)$ is a log resolution of $(X, D)$, so that we have
\[
K_{Y} + D_Y= f^*(K_X + D).
\]
It is known that the PL-homeomorphism class of $\mathcal{D}(D_Y^{=1})$ does not depend on the choice of a log resolution
, see \cite[Proposition 11]{dFKX17}.

\subsection{Calabi-Yau pairs and coregularity}
\label{sec-CY-coreg}
Let $(X, D)$ be an lc sub-pair. We say $(X,D)$ is a \emph{log Calabi-Yau sub-pair} (or \emph{log CY sub-pair} for short) if $D$ is a  Weil divisor with integral coefficients, and $K_X+D$ is linearly equivalent to $0$, denoted by, $K_X + D\sim 0$\,\footnote{Note that our definition is more restrictive than the usual one, where $D$ could have fractional coefficients, and  $K_X+D$ is $\mathbb{Q}$-linearly trivial.}. 
A log CY sub-pair $(X, D)$ is called a \emph{log Calabi-Yau pair} (or \emph{log CY pair}) if $D$ is effective. 
We call a log Calabi-Yau pair $(X,D)$ a {\em toric pair} if $X$ is a $\Bbbk$-split toric variety and the boundary $D$ is a torus invariant divisor. We refer to the pair $(\bP^n, \Delta_n)$ with $\Delta_n=\{x_1x_2\cdots x_{n+1}=0\}$ as the {\em standard toric pair}. 

The \emph{coregularity} of a log Calabi-Yau pair $(X, D)$, denoted by $\mathrm{coreg}(X, D)$ and defined in \cite{Mor22}, is the dimension of a minimal lc-center on any snc (or dlt) modification $(Y, D_Y)$ of $(X, D)$. Note that $\mathrm{coreg}(X, D)$ is equal to the minimal dimension of a stratum of $D_Y^{=1}$. 
Equivalently, the coregularity is equal to 
\[
\mathrm{coreg}(X, D)=\dim X-\dim \mathcal{D}(X, D)-1,
\]
where $\mathcal{D}(X, D)$ is the dual complex of the pair $(X, D)$. 
For more details on coregularity, see \cite{Mor22}. 
The following lemma is a generalization of \cite[Lemma 4.2]{ALP24}. 

\begin{lemm}\label{lemm:corregcubic}
    Let $S\subset \bP^3$ be an irreducible cubic surface with at worst du Val singularities, and $H\subset\bP^3$ a plane. Assume one of the following holds:
    \begin{itemize}
        \item  $S\cap H$ is the union of a line and a smooth conic intersecting transversally;
        \item $S\cap H$ is the union of three lines forming a triangle.
    \end{itemize}
    Then the pair $(\bP^3, S+H)$ is a log Calabi-Yau pair with coregularity 0.
\end{lemm}
\begin{proof}
Put $D=S+H$. 
First, we check that $(\mathbb{P}^3, D)$ is an log Calabi-Yau pair. 
The linear equvalence $K_{\mathbb{P}^3}+D\sim 0$ is obvious. 
By assumption, the pair $(H, S|_H)$ is lc. By inversion of adjunction,  $(\bP^3, D)$ is lc near $S\cap H$. Since $S$ has canonical singularities and $H$ is smooth, we know that the pair $(\bP^3, D)$ is lc. 
 
Let $f\colon Y\to \mathbb{P}^3$ be a log resolution of $(\mathbb{P}^3, D)$, and $(Y,  D_Y)$ the log pullback of $(\mathbb{P}^3, D)$. Note that $D_Y$ is a sub-boundary. 
To conclude that $\mathrm{coreg}(\bP^3, D)=0$, it suffices to check that $D^{=1}_Y$ admits a $0$-dimensional stratum. 

Let $\widetilde{H}$ be the strict preimage of $H$ via the map $f$. Put $D_{\widetilde{H}}=(D_Y-\widetilde{H})|_{\widetilde{H}}$. 
Note that the snc Calabi-Yau sub-pair $(\widetilde{H},D_{\widetilde{H}})$ is the log pullback of $(H, S|_H)$, where $H\simeq \mathbb{P}^2$ and $S|_H$ is either the union of three lines forming a triangle, or the union of a line and a conic intersecting transversally. Since $\mathrm{coreg}(H, S|_H)=0$, it follows that $\mathrm{coreg}(\widetilde{H},D_{\widetilde{H}})=0$. Thus, the snc divisor $D_{\widetilde{H}}^{=1}$ admits a zero-dimensional stratum. Hence, the divisor $D^{=1}_Y$ admits a zero-dimensional stratum as well. This shows that $\mathrm{coreg}(Y, D_Y)=0$, and so $\mathrm{coreg}(\mathbb{P}^3, D)=0$ as claimed.
\end{proof}
\subsection{Burnside groups}\label{sect:deficrepant}
We recall the definitions of several versions of Burnside groups introduced in \cite{KTs} and \cite{CLKT}, and introduce the definition of the {\em divisorial Burnside group}.

\begin{defi}
The \emph{Burnside group} $\mathrm{Burn}_n(\Bbbk)$ is a free abelian group generated by the $\Bbbk$-birational isomorphism  classes $[X]$ of $n$-dimensional (reduced and $\Bbbk$-irreducible) algebraic varieties $X$ over $\Bbbk$.
\end{defi}

\begin{defi}\label{defi:vol}
 Consider pairs
$
[X,\omega_X]
$ where
\begin{itemize}
    \item $X$ is an $n$-dimensional (reduced and $\Bbbk$-irreducible) smooth and proper algebraic variety over $\Bbbk$, and
    \item $\omega_X$ is a rational volume form  on $X$, i.e., $\omega_X\in \Omega^n_{\Bbbk(X)}$.
\end{itemize}
Two pairs $[X,\omega_X]$ and $[Y,\omega_Y]$ with volume forms $\omega_X$ and $\omega_Y$ are equivalent if there exists a birational map $\phi\colon X\dashrightarrow Y$ such that $\phi^*(\omega_Y)=\omega_X$. When such a map exists, we say $\phi$ is a \emph{volume preserving} birational map, and the pair $[Y,\omega_Y]$ is a {\em model} of $[X,\omega_X]$. The existence of a volume preserving map is also equivalent to the existence of a diagram 
\begin{equation}
\begin{tikzcd}
& \left[Z, \omega_Z\right] \ar[rd, "g"] \ar[dl, swap, "f"] & \ \\
\left[X, \omega_X\right] \ar[rr, dashed, "\phi"] & & \left[Y, \omega_Y\right]
\end{tikzcd}
\end{equation}
where $\phi$ is a birational map, $f$ and $g$ are birational contractions and 
\[
\omega_Z=f^*(\omega_X)=g^*(\omega_Y).
\]

\end{defi}

\begin{defi}\label{defi:logform}
The \emph{Burnside group of logarithmic volume forms} $\Burnf_n(\Bbbk)$ is a free abelian group generated by the equivalence classes of pairs $[X, \omega_X]$ where
\begin{itemize}
    \item $X$ is an $n$-dimensional (reduced and $\Bbbk$-irreducible) smooth and proper algebraic variety over $\Bbbk$, and
    \item $\omega_X$ is a {\em logarithmic} volume form on $X$, which means, $\omega_X\in\Omega^n_{\Bbbk(X)}$, and
for all proper smooth models $[Y,\omega_Y]$ of $[X,\omega_X]$ such that the divisor of zeros and poles of $\omega_Y$ has simple normal crossings, the rational differential form $\omega_Y$ has poles of order at most $1$.
\end{itemize}
\end{defi}

There is a natural forgetful homomorphism 
\begin{equation}\label{eq:forget}
  \varrho\colon\Burnf_n(\Bbbk)\to \Burn_n(\Bbbk)
\end{equation}
given by 
$$
[X, \omega_X]\mapsto [X],
$$ 
and a natural embedding  
\begin{equation}
\label{eq-iota}
\iota\colon \Burn_n(\Bbbk) \hookrightarrow\Burnf_n(\Bbbk)
\end{equation}
given by 
$$
[X]\mapsto [X, 0].
$$ 

\begin{defi}
\label{diag-crepant-equivalence}
Two sub-pairs $(X, D_X)$ and $(Y, D_Y)$ are \emph{crepant equivalent} if there exists the following diagram 
\begin{equation}
\begin{tikzcd}
& \left(Z, D_Z\right ) \ar[rd, "g"] \ar[dl, swap, "f"] & \ \\
\left(X, D_X\right ) \ar[rr, dashed, "\phi"] & & \left(Y, D_Y\right )
\end{tikzcd}
\end{equation}
where $\phi$ is a birational map, $f$ and $g$ are birational contractions, and   
\[
f_*D_Z = D_X, \quad \quad g_*D_Z = D_Y, \quad \quad K_Z+D_Z=f^*(K_X+D_X)=g^*(K_Y+D_Y).
\]
Under these conditions, we say $\phi$ is a {\em crepant birational map}, or a \emph{crepant equivalence}. 
\end{defi}

\begin{exam}
Consider two pairs $(X, D_X)$ and $(Y, D_Y)$ where $X$ and $Y$ are normal toric varieties of dimension $n$, and the divisors $D_X$ and $D_Y$ are sums of all torus-invariant divisors with coefficient $1$ on $X$ and $Y$, respectively. 
Then the two pairs are crepant equivalent lc pairs.
\end{exam}

Finally, we introduce the following definition.
\begin{defi}
The {\em divisorial Burnside group}  $\textbf{DivBurn}_n(\Bbbk)$ is a free abelian group generated by the crepant equivalence classes of lc sub-pairs $(X, D)$ where 
\begin{itemize}
    \item $X$ is a $n$-dimensional (reduced and $\Bbbk$-irreducible) smooth and proper algebraic variety over $\Bbbk$, and
    \item $D=\sum a_i D_i$ is a divisor with $a_i\in \mathbb{Z}$ (note that $a_i$ could be negative).
\end{itemize} 
Note that from the assumption that $(X,D)$ is lc, it follows that $a_i\leq 1$, and this holds for any model of $(X, D)$. 
\end{defi}

Let $\Burnf_{n}^{\ne 0}(\Bbbk)$ be the subgroup of $\Burnf_n(\Bbbk)$ generated by pairs $[X,\omega_X]$ such that $\omega_X\ne0$, and  $\Burnf_{n}^{= 0}(\Bbbk)$ the subgroup generated by pairs of the form $[X,0].$ We have  
$$
\Burnf_n(\Bbbk)=\Burnf_{n}^{\ne 0}(\Bbbk)\oplus\Burnf_{n}^{= 0}(\Bbbk).
$$
Note that $\iota(\Burn_n(\Bbbk))=\Burnf_{n}^{= 0}(\Bbbk)$, where $\iota$ is as in \eqref{eq-iota}. 
Consider the map
\begin{equation}
\label{eq-map-delta}
\delta_n\colon \Burnf_n^{\neq 0}(\Bbbk)\to \textbf{DivBurn}_n(\Bbbk)
\end{equation}
given by
$$
[X, \omega_X] \mapsto (X, -\mathrm{div}(\omega_X)),
$$
where $\mathrm{div}(\omega_X)$ is the divisor of zeroes and poles of $\omega_X$. 
Note that the pair $(X, -\mathrm{div}(\omega_X))$ is lc  by the assumption on the poles of $\omega_X$. 
It follows that the pair $(X, -\mathrm{div}(\omega_X))$ is a log Calabi-Yau sub-pair. 

Vice versa, given a log Calabi-Yau sub-pair $(X, D)$, we have $K_X\sim -D$. So there exists a rational logarithmic volume form on $X$ whose divisor of zeroes and poles is $-D$. This proves the following.

\begin{prop}\label{prop:lcy}
The image of the $\delta_n$ map \eqref{eq-map-delta} in $\mathbf{DivBurn}_n(\Bbbk)$ coincides with the subgroup generated by log Calabi-Yau sub-pairs.
\end{prop}
 
Note that $\delta_n$ is not surjective even for $n=1$, since not any lc sub-pair is Calabi-Yau. Clearly, $\delta_n$ is not injective as well, because multiplying a volume form by a non-zero constant does not change its divisor of zeroes and poles. Also, $\delta_n$ cannot be naturally extended to $\Burnf^{=0}(\Bbbk)$ as the following example shows.

\begin{exam}

Let $f\colon \mathbb{F}_1\to \bP^2$ be the blow up of a point on $\bP^2$. Denote by $E$ the $f$-exceptional divisor. 
Then the pair $[\bP^2, 0]$ is equivalent to the pair $[\mathbb{F}_1, 0]$ in $\Burnf_2(\Bbbk)$. 
On the other hand, the pair $(\bP^2, 0)$ is equivalent to the pair $(\mathbb{F}_1, -E)$ in $\textbf{DivBurn}_2(\Bbbk)$. One can check that the latter pair is not equivalent to $(\mathbb{F}_1, 0)$.
\end{exam}

\subsection{Pluri-canonical representation}
Let $[X, \omega]\in\Burnf_n(\Bbbk)$ with  $\omega\neq 0$. We denote by $\Bir_\Bbbk(X,\omega)$ the group of volume preserving birational automorphisms over $\Bbbk$
$$
[X,\omega]\dashrightarrow[X,\omega].
$$
Consider the associated log CY sub-pair $(X, D)$ with $D=-\mathrm{div}(\omega)$. Assume that $(X, D)$ is a log CY pair, so that $K_X+D\sim 0$ and $D\geq 0$. Let 
    $\Bir_\Bbbk(X, D)$ be the group of crepant birational automorphisms over $\Bbbk$ of $(X,D)\dashrightarrow(X,D)$.
    
Here we explain how the groups $\Bir_{\Bbbk}(X, \omega)$ and $\Bir_{\Bbbk}(X, D)$ are related. First, note that a birational automorphism can act on a logarithmic volume form non-trivially while preserving the corresponding divisor of its zeroes and poles:


\begin{exam}
Let $\sigma_n\colon \bP^n\dashrightarrow\bP^n$ be the Cremona involution given by the formula 
\[
(x_0:\ldots: x_n)\mapsto \left(\frac{1}{x_0}:\ldots:\frac{1}{x_n}\right).
\]
Then $\sigma_n$ acts on the standard torus invariant volume form $\omega_n$ on $\bP^n$ via multiplication by $(-1)^n$, and preserves the divisor of its zeroes and poles.
\end{exam}

For a log Calabi-Yau pair $(X, D)$, there exists a (unique up to a scalar) logarithmic volume form that corresponds to a non-zero element in $\rH^0(X, \cO_X(K_X+D))$. The group $\Bir_{\Bbbk}(X, D)$ acts linearly on the $1$-dimensional vector space $\rH^0(X, \cO_X(K_X+D))$, cf. \cite[2.14]{HaXu}. This gives rise to a representation 
\[
\rho = \rho_{(X, D)}\colon \Bir_{\Bbbk}(X, D) \to \mathsf{GL}(\rH^0(X, \cO_X(K_X+D)))
\]
called the \emph{(pluri-)canonical representation}\,\footnote{In fact, this map is defined in a more general setting when $m(K_X+D)\sim 0$ for some $m>0$, and then $\mathrm{Bir}_{\Bbbk}(X, D)$ acts by linear automorphisms on the space of pluri-forms $\mathrm{H}^0(X, \cO_X(m(K_X+D)))$.} of $\Bir_{\Bbbk}(X, D)$. 
The following result is a generalization of the classical theorem due to Ueno and Deligne:
\begin{theo}[{\cite[Theorem 1.1]{FG14}, \cite[Theorem 1.2]{HaXu}}]
\label{theo:fujino}
Let $(X, D)$ be a log CY pair. Then the image of the pluri-canonical representation $\rho_{(X, D)}(\Bir_{\Bbbk}(X, D))$ 
is finite.
\end{theo}
Now let $[X,\omega]\in \Burnf_n(\Bbbk)$, and let $(X, D)$ be a log CY sub-pair such that $D=-\mathrm{div}(\omega)$. 
Assume that $D$ is effective. 
Then it follows from Theorem~\ref{theo:fujino} that we have an exact sequence 
\[
1\to \Bir_{\Bbbk}(X, \omega) \to \Bir_{\Bbbk}(X, D)\to \mathrm{Im}(\rho) \to 1
\]
where $\mathrm{Im}(\rho)$ is a finite group. This implies the following corollary.

\begin{coro}\label{coro:volumepreserve}
Let $[X,\omega]$ be a pair in $\Burnf_n(\Bbbk)$ and put $D=-\mathrm{div}(\omega)$. Assume that $(X, D)$ is a log Calabi-Yau pair, in particular, $D$ is effective.
Then for any $\psi\in \Bir_{\Bbbk}(X, D)$, there exists some positive integer $N$ such that $\psi^N\in \Bir_{\Bbbk}(X, \omega)$.
\end{coro}

Theorem~\ref{theo:fujino} and Corollary~\ref{coro:volumepreserve} show that for log Calabi-Yau pairs, the notions of crepant birational automorphisms and volume preserving maps are equivalent up to raising the automorphism to some finite power. 
We use this fact in Section~\ref{sect:ellip} and ~\ref{sect:K3}. 

\subsection{Burnside rings and stable equivalence}\label{sect:stablerat}
There exists a ring structure on the Burnside group 
\[
\Burnf(\Bbbk)=\bigoplus_{n\geq 0}\Burnf_n(\Bbbk)
\]
with the multiplication given by the following: for $[X, \omega_X]$ and $[Y, \omega_Y]$ in $\Burnf_n(\Bbbk)$ and $\Burnf_m(\Bbbk)$, we put
\[
[X, \omega_X]\cdot [Y, \omega_Y] = [X\times Y, \omega_X\wedge\omega_Y]. 
\]
Then we extend this definition by linearity. 
Similarly, we define a ring structure on 
\[
\textbf{DivBurn}(\Bbbk)=\bigoplus_{n\geq 0}\textbf{DivBurn}_n(\Bbbk)
\]
as follows. 
For $(X, D_X)$ and $(Y, D_Y)$ in $\textbf{DivBurn}_n(\Bbbk)$ and $\textbf{DivBurn}_m(\Bbbk)$, we put
\[
(X, D_X)\cdot (Y, D_Y) = (X\times Y, \pi_1^* D_X + \pi_2^* D_Y)
\]
where $\pi_1\colon X\times Y \to X$ and $\pi_2\colon X\times Y \to Y$ are the canonical projections to the first and second factor. 
Then we extend this definition by linearity. 
Note 
that the map $\delta=\sum_{n\geq 0} \delta_n$ where $\delta_n$ is as in \eqref{eq-map-delta} becomes a ring homomorphism  
\[
\delta\colon \Burnf^{\neq 0}(\Bbbk)=\bigoplus_{n\geq 0}\Burnf^{\neq 0}_n(\Bbbk)\to \textbf{DivBurn}(\Bbbk).
\] 
In light of this, we define a map 
$$
j_{n,r}:\dBurn_{n}(\Bbbk)\to\dBurn_{n+r}(\Bbbk)
$$
given by 
$$
(X,D)\mapsto (X\times\bP^r, \pi_1^* D+ \pi_2^*\Delta_r)
$$
where $\Delta_r=\{x_1x_2\cdots x_{r+1}=0\}\subset\bP^r_{x_1,\ldots,x_{r+1}}$ is the standard toric boundary on $\bP^r$.   Given two pairs 
$$
(X,D_X)\in\dBurn_{n}(\Bbbk)\quad\text{and}\quad (Y,D_Y)\in\dBurn_{m}(\Bbbk),
$$ we say they are {\em stably crepant birational} if there exists $r_1,r_2\in\bN$ such that 
$$
j_{n,r_1}(X,D_X)=j_{m,r_2}(Y,D_Y).
$$
Nontrivial stable birationalities between algebraic varieties have been exhibited in \cite{BCTSWDrat}. It would be interesting to study nontrivial stable crepant birationalities of pairs. We propose the following:
\begin{ques}
Do there exist pairs $(X, D_X)$ and $(Y, D_Y)$ of coregularity $0$ such that they are not crepant birational, but stably crepant birational?
\end{ques}
Note that coregularity of a pair is invariant under crepant birational maps, and if $\mathrm{coreg}(X, D)=0$ and $\dim X=n$, then $\mathrm{coreg}\,j_{n,r}(X,D)=0$ for any $r\geq0$. 

\section{Invariants of birational maps}
\label{subsec-birational-invariants}
In \cite{SL}, for any birational map $\phi\colon X\dashrightarrow Y$ over $\Bbbk$ with $\dim(X)=n$, an invariant~$c(\phi)$ taking values in $\Burn_{n-1}(\Bbbk)$ is defined as follows. 
Consider a resolution of the map $\phi$:
\begin{equation*}
\begin{tikzcd}
& Z \ar[rd, "g"] \ar[dl, swap, "f"] & \ \\
X \ar[rr, dashed, "\phi"] & & Y
\end{tikzcd}
\end{equation*}
Then put
\[
{c}(\phi)  = \sum_i[E_i] - \sum_j[F_j]\in \Burn_{n-1}(\Bbbk)
\]
where the sum $\sum_i [E_i]$ runs over $\Bbbk$-irreducible components of the $f$-exceptional locus $\mathrm{Exc}(f)$, the sum $\sum_j [F_j]$ runs over $\Bbbk$-irreducible components of the $g$-exceptional locus $\mathrm{Exc}(g)$. One can  check that $c(\phi)$ does not depend on the choice of a resolution, and $c$ induces a group homomorphism \cite[Lemma 2.2]{SL}
$$
c\colon\Bir_\Bbbk(X)\to\Burn_{n-1}(\Bbbk).
$$

A similar invariant $\mathbf c$ is defined for volume preserving birational maps in \cite{CLKT}. Let $[X, \omega_X], [Y, \omega_Y]$ be two representatives of an equivalence class in $\Burnf_n(\Bbbk)$, and
$$
\phi\colon [X, \omega_X]\dashrightarrow [Y, \omega_Y]
$$
a volume preserving birational map. 
Then there exists a resolution
\begin{equation}
\label{diagram-c-bir}
\begin{tikzcd}
& \left[Z, \omega_Z\right] \ar[rd, "g"] \ar[dl, swap, "f"] & \ \\
\left[X, \omega_X\right] \ar[rr, dashed, "\phi"] & & \left[Y, \omega_Y\right]
\end{tikzcd}
\end{equation}
where $f$ and $g$ are birational contractions and 
\[
\omega_Z=f^*(\omega_X)=g^*(\omega_Y).
\]
An invariant $\mathbf{c}(\phi)\in\Burnf_{n-1}(\Bbbk)$ is associated to $\phi$ as follows. 
Let $\cup_{i\in I} E_i$ be the union of $\Bbbk$-irreducible components of the $f$-exceptional locus $\mathrm{Exc}(f)$, and $\cup_{j\in J} F_j$ the union of $\Bbbk$-irreducible components of the $g$-exceptional locus $\mathrm{Exc}(g)$. We put 
\[
\mathbf{c}(\phi) =  \sum_{i\in I}[E_i, \rho_{E_i}(\omega_Z)] - \sum_{j\in J}[F_j, \rho_{F_j}(\omega_Z)]\in\Burnf_{n-1}(\Bbbk)
\]
where  $\rho$ is the residue map defined in \cite[Section 4]{CLKT}.
In \cite{CLKT}, it is proven that $\mathbf{c}(\phi)$ does not depend on the choice of a resolution $Z$. 



Let $[X, \omega_X]\in\Burnf_n(\Bbbk)$. The induced map 
$$
\mathbf{c}\colon \Bir_{\Bbbk}(X, \omega_X)\to \Burnf_{n-1}(\Bbbk)
$$
is a group homomorphism \cite[Corollary 7.6]{CLKT}. In the next lemma, we show that the map $\mathbf{c}$ on $\Bir_{\Bbbk}(X, \omega_X)$ essentially coincides with the map $c$ from \cite{SL}, so the presence of volume forms does not give much new information. 

\begin{lemm}\label{lemm:discr}
Let $[X, \omega_X]\in\Burnf_n(\Bbbk)$. Then the image $\mathbf c(\Bir_\Bbbk(X,\omega_X))$ is contained in the group $\Burnf^{=0}_{n-1}(\Bbbk)$.
Moreover, we have a commutative diagram
\begin{equation}
    \begin{tikzcd}
\Bir_\Bbbk(X,\omega_X)\ar[r,"\mathbf{c}"]\arrow[d]&\Burnf_{n-1}(\Bbbk)\\
          \Bir_\Bbbk(X)\ar[r,"c"]&\ar[u,"\iota"]\Burn_{n-1}(\Bbbk)
    \end{tikzcd}
\end{equation}
where the left vertical arrow is the natural inclusion and $\iota$ is as in \eqref{eq-iota}. 
\end{lemm}
\begin{proof}

It suffices to show that 
\[
\mathbf{c}(\phi)=\sum_i\pm[E_i, 0]\in \Burnf_{n-1}^{=0}(\Bbbk)\subset \Burnf_{n-1}(\Bbbk)
\] 
for some divisors $E_i$ over $X$. 
The claim is straightforward if $\omega_X=0$. So we may assume that $\omega_X\neq 0$. 
For any element $\phi\in \Bir_{\Bbbk}(X,\omega_X)$, consider the diagram 
\begin{equation}
\begin{tikzcd}
& \left[Z, \omega_Z\right] \ar[rd, "g"] \ar[dl, swap, "f"] & \ \\
\left[X, \omega_X\right] \ar[rr, dashed, "\phi"] & & \left[X, \omega_X\right]
\end{tikzcd}
\end{equation}
where $f$ and $g$ are birational contractions and 
$
\omega_Z=f^*(\omega_X)=g^*(\omega_X).
$
Put $D=-\mathrm{div}(\omega_X)$ and $D_Z=-\mathrm{div}(\omega_Z)$. 
Note that $D$ and $D_Z$ are integral divisors with coefficients at most $1$. Observe that the pair $(X, D)$ is lc, cf. Proposition~\ref{prop:lcy},
and the pair $(Z, D_Z)$ is the log pullback of $(X, D)$. 

We show that only divisors $E$ on $Z$ with  $a(E, X, D)>0$ can make non-trivial contributions to $\mathbf{c}(\phi)$, and such divisors only contribute to elements in $\iota(\Burn_{n-1}(\Bbbk)).$
Let $E$ be a divisor on $Z$ which is either $f$- or $g$-exceptional. If $E$ is both $f$- and $g$-exceptional, then it does not contribute to $\mathbf{c}(\phi)$. Let us assume that $E$ is $f$-exceptional and not $g$-exceptional. Note that $a(E, X, D)=a(E,Z,D_Z)$, because the pairs $(X, D)$ and $(Z, D_Z)$ are crepant equivalent. 

When $a(E, X, D)>  0$, we have $\mathrm{coeff}_E D_Z\leq 0$. This means that $\omega_Z$ is regular at the generic point of $E$. 
By definition of the residue map \cite[Section 4]{CLKT}, this implies that $\rho_E (\omega_Z)=0$. Thus, the class $[E, 0]$ in $\mathbf{c}(\phi)$ belongs to \[
\iota(c(\phi))\subset \iota(\Burn_{n-1}(\Bbbk))=\Burnf_{n-1}^{=0}(\Bbbk).\]

When $a(E, X, D)=0$, we have $\mathrm{coeff}_E D_Z=1$. Put $\omega_E=\rho_E(\omega_Z)$. 
We show that the class $[E, \omega_E]$ is cancelled in $\mathbf{c}(\phi)$ by some other term. Indeed, since $E$ is not $g$-exceptional and $g_*(D_Z)=D$, it follows that $g(E)$ is a component of $D$ with coefficient~$1$. Put $g(E)=D_1$ and  $\omega_{D_1}=\rho_{D_1}(\omega_X)$.


From the definition of the residue map \cite[Section 4]{CLKT}, it follows that taking a residue of the form $\omega_X$ commutes with pullback: 
\begin{equation}
\label{eq-pullback-residue}
(g|_E)^*(\omega_{D_1})=(g|_{E})^*(\rho_{D_1}(\omega_X))=\rho_E(g^*(\omega_X)) = \rho_E(\omega_Z) = \omega_E. 
\end{equation}
It follows that $[E, \omega_E]=[D_1,\omega_{D_1}]$ in $\Burnf_{n-1}(\Bbbk)$. 

Let $D_1,\ldots, D_k$ be all the components of $D$ with coefficient $1$ such that $[D_i, \omega_{D_i}]=[D_1, \omega_{D_1}]$ in $\Burnf_{n-1}(\Bbbk)$ for $\omega_{D_i}=\rho_{D_i}(\omega_X)$. Consider their strict transforms $D'_i$ on $Z$ which belong to $D_Z$ with coefficient $1$. 
Put $\omega_{D'_i}=\rho_{D'_i}(\omega_Z)$. 
Similarly to \eqref{eq-pullback-residue}, we have $(f|_{D'_i})^*(\omega_{D_i})=\omega_{D'_i}$. Thus, all the classes $[D'_i, \omega_{D'_i}]=[E, \omega_E]$ coincide in $\Burnf_{n-1}(\Bbbk)$ for $1\leq i\leq k$. 

We have $g_*(D_Z)=D$, and $D^{=1}$ contains $k\geq 1$ divisors whose classes are equal to $[D_i,\omega_{D_i}]$. On the other hand, $D_Z^{=1}$ contains at least $k+1$ divisors with such classes, and at least one of them, $E_i$, is $f$-exceptional and not $g$-exceptional. 
It follows that at least one of $D'_i$, say $D'_2$, should be contracted by $g$. 
We conclude that the class $[E, \omega_E]$ is cancelled by $[D'_2,\omega_{D'_2}]$ in $\mathbf{c}(\phi)$. 

So we can cancel the terms $[E, \omega_E]-[D'_2,\omega_{D'_2}]$ from $\mathbf{c}(\phi)$.  
Proceeding to the next $f$-exceptional and not $g$-exceptional divisor and using induction, and then applying a similar argument to $g$-exceptional and not $f$-exceptional divisor, we conclude that 
\[
\mathbf{c}(\phi)=\sum_i\pm[E_i, 0]\in \Burnf_{n-1}(\Bbbk)
\] 
for some divisors $E_i$ on $Z$. 
Thus, the image $\mathbf{c}(\Bir_{\Bbbk}(X,\omega_X))$ is contained in $\iota(\Burn_{n-1}(\Bbbk))$. This completes the proof.
\end{proof}


\begin{coro}
\label{prop-reg0-vanishing} 
In each of the following cases, we have $\mathbf{c}(\Bir_{\Bbbk}(X, \omega_X))=0$:
\begin{itemize}
\item
$[X, \omega_X]\in \Burnf_2(\Bbbk)$ where $\Bbbk$ is arbitrary, or 
\item 
$[X, \omega_X]\in \Burnf_3(\Bbbk)$ where $\Bbbk$ is algebraically closed.
\end{itemize}
\end{coro}
\begin{proof}
By Lemma~\ref{lemm:discr}, it suffices to show that $c(\Bir_{\Bbbk}(X))=0$ under these assumptions. This follows from \cite[Theorem 2.5, Proposition 2.6]{SL}. 
\end{proof}

Abusing notation, we also define an invariant $\mathbf{c}$ for a crepant birational map
\[
\phi\colon (X, D_X)\dashrightarrow(Y, D_Y)
\]
in a similar way as above. Recall that by definition, there exists a diagram 
\begin{equation*}
\begin{tikzcd}
& \left(Z, D_Z\right ) \ar[rd, "g"] \ar[dl, swap, "f"] & \ \\
\left(X, D_X\right ) \ar[rr, dashed, "\phi"] & & \left(Y, D_Y\right )
\end{tikzcd}
\end{equation*}
where $\phi$ is a birational map, $f$ and $g$ are birational contractions, and   
\[
f_*D_Z = D_X, \quad \quad g_*D_Z = D_Y, \quad \quad K_Z+D_Z=f^*(K_X+D_X)=g^*(K_Y+D_Y).
\] 
Let $\cup_{i\in I} E_i$ be the union of irreducible components of $f$-exceptional divisors, and $\cup_{j\in J} F_j$ the union of irreducible components of $g$-exceptional divisors. We define 
\[
\mathbf{c}(\phi) =\sum_{i\in I}(E_i, D_{E_i}) - \sum_{j\in J}(F_j, D_{F_j})\in\mathbf{DivBurn}_{n-1}(\Bbbk),
\]
where $D_{E_i}$ is defined in the following way. We put
\[
D_{E_i} =\begin{cases}
    -K_{E_i} + (K_Z + D_Z)|_{E_i}&\text{if $E_i$ belongs to $D_Z$ with coefficient $1$,}\\
    0&\text{otherwise,}
\end{cases} 
\]
and $D_{F_j}$ is defined similarly. A similar proof of Lemma \ref{lemm:discr} shows the following.

\begin{lemm}\label{lemm:cdiv}
Let $(X,D_X)\in\mathbf{DivBurn}_{n-1}(\Bbbk)$, and $\phi\in \Bir_{\Bbbk}(X, D_X)$. Then $\mathbf{c}(\phi)$ belongs to the subgroup in $\mathbf{DivBurn}_{n-1}(\Bbbk)$ generated by pairs of the form $(E, 0)$ where $\dim E=n-1$. 
\end{lemm}

\begin{rema}\label{rema:samec}
    Note that we abuse the notation $\bc$ to refer to the invariants of both volume preserving and crepant birational automorphisms. This is due to the fact that both invariants can be essentially identified with the Lin-Shinder $c$-invariant of the underlying birational maps, as Lemma~\ref{lemm:discr} and Lemma~\ref{lemm:cdiv} imply. In other words, the presence of volume forms and boundaries does not provide much additional information in the invariants. This simplifies the computation of the $\mathbf c$-invariant of various crepant birational maps in Section~\ref{sect:ellip} and~\ref{sect:K3}.
\end{rema}

\begin{exam}[Pseudo-regularizable map]\label{exam:vanishpseudo}
   A birational map $\phi\in\Bir_\Bbbk(X)$ is called {\em pseudo-regularizable} if $\phi=\alpha^{-1}\circ \gamma\circ \alpha$ where 
    \begin{itemize}
        \item $\alpha:X\dashrightarrow X'$ is a birational map with some variety $X'$ over $\Bbbk$, and 
        \item $\gamma\in\Bir_\Bbbk(X')$ is an isomorphism in codimension 1.
    \end{itemize}
 In particular, pseudo-regularizable maps include regular automorphisms and birational automorphisms of finite order, see \cite[Example 4.2]{SL}. We call a volume preserving birational automorphism $\phi\in \Bir_\Bbbk(X,\omega)$ pseudo-regularizable if $\phi$ is pseudo-regularizable as an element in $\Bir_\Bbbk(X).$ For such maps $\phi,$ it is not hard to see that $c(\phi)=0$ since $c(\gamma)=0$ by definition. Then it follows from Lemma~\ref{lemm:discr} that $\mathbf{c}(\phi)=0.$
\end{exam}

\

\section{$\bP^3$ with quintic genus one curves}\label{sect:ellip}
In this section, we construct a crepant birational automorphism of the standard toric pair over certain field $\Bbbk$
$$
\sigma\colon (\bP^3,\Delta_3)\dashrightarrow (\bP^3,\Delta_3),\quad \Delta_3=\{x_1x_2x_3x_4=0\}\subset\bP^3_{x_1,\ldots,x_4}
$$
such that 
$$
0\ne \mathbf{c}(\sigma)\in \mathbf{DivBurn}_2(\Bbbk). 
$$
Our construction is a refinement of \cite[Proposition 3.6]{SL}, which relies on the classical quadro-cubic Cremona transformation \cite[Chapter VII, 5.1]{sempleroth}. We briefly recall their construction. 

Let $C$ be a quintic genus one curve, given by a scheme-theoretic intersection of five quadrics in $\bP^4$. Let $Q$ be a smooth quadric threefold in $\bP^4$ containing $C$. The linear system of quadrics in $\bP^4$ containing $C$ yields a birational map
\begin{align}\label{eqn:varphidefi}
\varphi:Q\dashrightarrow \bP^3.
\end{align}
The map $\varphi$ blows up $C$ and then blows down a divisor to a curve $C'$ in $\bP^3$. 
The curve $C'$ is isomorphic to the Jacobian $\mathrm{Jac}^2(C)$ of degree 2 line bundles on $C$. 
The inverse map $\varphi^{-1}$ is given by the linear system of cubics passing through $C'$. 
By \cite{SL}, $Q$ always contains a $\Bbbk$-line. Projection from a point on the line induces a birational map 
$$
\pi\colon Q\dashrightarrow\bP^3.
$$ 
The map $\varphi\circ\pi^{-1}$ is then a birational automorphism of $\bP^3.$  When $C$ has no $\Bbbk$-points, $C$ and $C'$ are not isomorphic, and one has
$$
0\ne c( \varphi\circ\pi^{-1})=[C\times\bP^1]-[\mathrm J^2(C)\times\bP^1]\in\Burn_2(\Bbbk).
$$

Now, the goal is to extend this map to a crepant birational map of some log CY pairs that admit toric models. 
\begin{prop}\label{prop:ellip}
   Let $\Bbbk$ be a field of one of the following:
   \begin{itemize}
       \item a number field;
       \item the function field of an algebraic variety over a number field, over a finite field, or over an algebraically closed field.
   \end{itemize}
    Let $C\subset\bP^4$ be a quintic genus one curve over $\Bbbk$ with no $\Bbbk$-rational points. Then there exist crepant birational maps $\varphi, \pi,\psi,\psi',\eta,\eta'$ and
   $$
\sigma:=\eta'\circ\psi'\circ\varphi\circ\pi^{-1}\circ\psi^{-1}\circ\eta^{-1}
   $$
   with the following commutative diagram
    \begin{equation}\label{dia:ellip}
\begin{tikzcd}
(C\subset Q, Q_1+H_1)\ar[rr, "\varphi", dashed]\ar[dd,"\pi",dashed]&  &(C'\subset\bP^3, H_2+S_1) \ar[dd ,"\psi'"',dashed] \\
\\
(\bP^3,S_2+H_3)\ar[dd,"\psi", dashed]&&(\bP^3,S_4+H_5)\ar[dd, "\eta'"', dashed]\\
\\
(\bP^3,H_4+S_3)\ar[rd, "\eta",dashed]
&&(\bP^3,\Delta_3)\\
&(\bP^3,\Delta_3)\ar[ru, "\sigma", dashed]&
\end{tikzcd}
\end{equation}
where
\begin{itemize}
    \item $Q$ is a smooth quadric threefold containing $C$,
    \item $C'$ is $\Bbbk$-isomorphic to $\mathrm{Jac}^2(C)$, which is a quintic genus one curve not $\Bbbk$-isomorphic to $C$,
    \item $Q_1$ is a smooth quartic del Pezzo surface,
    \item $H_1$ is a quadric surface cone, i.e., a cone over a smooth conic,
    \item $S_2$ and $S_1$ are cubic surfaces with one $\sA_1$-singularity,
    \item $S_3$ and $S_4$ are cones over the union of a line and a smooth conic intersecting transversally,
    \item $H_i$ are planes in $\bP^3$ for $i=2,3,4,5$,
    \item $(\bP^3,\Delta_3)$ is the standard toric pair.
\end{itemize}
In particular, $\sigma$ is a crepant birational automorphism of the standard 3-dimensional toric pair $(\bP^3,\Delta_3)$ with
$$
0\ne\mathbf c(\sigma)=\mathbf c(\varphi) \in \mathbf{DivBurn}_2(\Bbbk).
$$

\end{prop}

\

The rest of the section is devoted to a constructive proof of Proposition~\ref{prop:ellip}. We construct the crepant birational maps in the diagram \eqref{dia:ellip} in different subsections.

\subsection{Constructing the boundary}
First, note that under our assumption of $\Bbbk$, there indeed exist quintic genus one curves with no $\Bbbk$-points \cite[Lemma 3.8]{SL}. For any such curve $C,$ let $Q$ be a general smooth quadric threefold containing $C$, and $\varphi$ the birational map \eqref{eqn:varphidefi} between $Q$ and $\bP^3$. To extend $\varphi$ to a map between log CY pairs, we find the boundary divisors as follows:
\begin{enumerate}
    \item Pick a general plane $H_2$ in $\bP^3$. The intersection $H_2\cap C'$ is a $\Gal(\bar \Bbbk/\Bbbk)$-invariant set consisting of five $\overline\Bbbk$-points. Since $H_2$ is chosen generally, we may assume that the five points are in general position on $H_2$. 
    \item Let $Q_1$ be the strict transform of $H_2$ under $\varphi^{-1}$. Then $Q_1$ is a smooth quartic del Pezzo surface (cf. \cite[Chapter VII, 5.2]{sempleroth}).
    \item  The strict transform  under $\varphi$ of the unique conic passing through the five points in $H_2\cap C'$ is a $\Bbbk$-line $l$ in~$Q_1$. Pick a general $\Bbbk$-point $q$ on $l$. Let $H_1$ be the quadric cone consisting of all lines in $Q$ passing through $q$. 
    \item Let $S_1$ be the image of $H_1$ under $\varphi$. 
    Recall that $H_1$ is a hyperplane section of $Q$, so the surface $S_1$ is a cubic surface in $\mathbb{P}^3$. 
    Since $q$ is chosen generally, we know that $S_1$ is the blowup of $H_1$ in five general points. It follows that $S_1$ is a cubic surface with one $\sA_1$-singularity. 
\end{enumerate}
Note that under these choices, the curve $C$ is contained in $Q_1$ but not in $H_1$, and $C'$ is in $S_1$ but not in $H_2$. It follows that $\varphi$ extends to a crepant birational map 
\begin{align*}
    \varphi\colon (Q,Q_1+H_1)\dashrightarrow (\bP^3,H_2+S_1)
\end{align*}
with
$$
0\ne\mathbf{c}(\varphi)=(C\times\bP^1,0)-(C'\times\bP^1,0)\in\mathbf{DivBurn}_2(\Bbbk).
$$

\subsection{Constructing the map $\pi$} Note that $Q_1\cap H_1$ is a curve of degree $4$ and arithmetic genus $1$ containing a $\Bbbk$-line $l$. By generality assumptions on $H_2$ and $Q_1$, we know that $Q_1\cap H_1$ is a union of $l$ and a twisted cubic $R$ meeting at two points. Moreover, one of the two points is the vertex of the quadric cone $H_1$ since all twisted cubics in a quadric cone pass through its vertex. It follows that the the other intersection point of $l\cap R$ is a $\Bbbk$-rational smooth point $p$ on $H_1$. Again, by the generality assumptions, $p$ belongs to only one line on $Q_1$. Let $\pi$ be the projection map from $p.$ We obtain a crepant birational map over $\Bbbk$:
$$
\pi\colon (Q,Q_1+H_1)\dashrightarrow(\bP^3,S_2+H_3).
$$
The map $\pi$ birationally transforms $Q$ to $\bP^3$ and the quadric cone $H_1$ to a plane $H_3$. The image of $Q_1$ under $\pi$ is a cubic surface $S_2$ with an $\sA_1$-singular point $q_2$. Indeed, the line $l$ is a $(-1)$-curve in $Q_1$. When restricted to $Q_1$, the map $\pi$ blows up the point $p$ and contract the strict transform of $l$ to $q_2$. The intersection $S_2\cap H_3$ consists of a smooth conic and a line meeting transversally at two $\Bbbk$-rational points. The conic is the image of the twisted cubic $R$ under $\pi$, and the line comes from the exceptional curve above $p.$ It follows from Lemma~\ref{lemm:corregcubic} that the pairs $(Q,Q_1+H_1)$ and $(\bP^3,S_2+H_3)$ are lc, and 
\begin{align*}
  \mathrm{coreg}(\bP^3,S_2+H_3)=0.
\end{align*}
As a birational map between $Q$ and $\bP^3$, $\pi$ is a composition of the blowup of $p$ and the contraction of the strict transform of $H_1$ in $\Bl_p(Q)$ to a conic. It follows that 
\begin{align}\label{eqn:trivcellippi}
    {\bf{c}}(\pi)=(\bP^2,0)-(\bF_2,0)=0.
\end{align}

\subsection{Constructing the maps $\psi$ and $\psi'$}
Recall that the line $l$ is contracted to the $\Bbbk$-rational singular point $q_2$ of $S_2$. It follows that $q_2\in S_2\cap H_3$ has multiplicity 3 on the divisor $S_2+H_3$. The following construction is due to Ducat \cite[Section 5.1]{Du24}. We present the details here to make sure it is defined over $\Bbbk.$ Up to a change of variables, we may assume $q_2=[1:0:0:0]$ and $S_2+H_3$ is given by equations
$$
S_2=\{x_1f_2+f_3=0\}\subset\bP^3_{x_1,x_2,x_3,x_4},\quad H_3=\{x_2=0\}\subset\bP^3_{x_1,x_2,x_3,x_4}
$$
where $f_2$ and $f_3$ are polynomials of degree 2 and 3 in variables $x_2,x_3,x_4$ over $\Bbbk.$ Note that since $q_2$ is an $\sA_1$-singularity, we know that $f_2$ is irreducible. 
Consider the map $\psi$ given by
\begin{align}\label{eqn:psbound}
    \psi\colon (\bP^3,S_2+H_3)&\dashrightarrow (\bP^3,H_4+S_3),\\
    (x_1: x_2: x_3: x_4)&\mapsto(x_1+\frac{f_3}{f_2}: x_2:x_3:x_4)\notag
\end{align}
where the new boundary divisor is given by 
$$
S_3=\{x_2f_2=0\},\quad H_4=\{x_1=0\}.
$$
We check (in the affine chart $x_4=1$) that $\psi$ preserves the volume forms corresponding to the respective boundary divisors in~\eqref{eqn:psbound}
$$
\psi^*\left(\frac{\dd x_1\wedge\dd x_2\wedge\dd x_3}{x_1x_2f_2}\right)=\frac{\dd (x_1+f_3f_2^{-1})\wedge\dd x_2\wedge\dd x_3}{(x_1+f_3f_2^{-1})x_2f_2}=\frac{\dd x_1\wedge\dd x_2\wedge\dd x_3}{x_2(x_1f_2+f_3)}.
$$
It follows that $\psi$ is a crepant birational map. In particular, $S_3\cap H_4$ is the union of a line and a smooth conic. Since the pair $(\bP^3,H_4+S_3)$ is lc, the line and the conic meet at two distinct points. Both of the two points are $\Bbbk$-rational since $S_2\cap H_3$ being the union of a line and a conic implies that $f_2(0,x_3,x_4)$ is reducible over $\Bbbk$. The inverse map of $\psi$ is given by 
$$
\psi^{-1}: (x_1: x_2: x_3: x_4)\mapsto(x_1-\frac{f_4}{f_3}: x_2:x_3:x_4).
$$
It is not hard to see that the exceptional divisors of $\psi$ and $\psi^{-1}$ are the same and thus 
\begin{align}\label{eqn:trivcpsi'ellip}
    \mathbf c(\psi)=0.
\end{align}
The map $\psi'$ is constructed in the same way: since the singular point $q$ of $H_1$ lies on $Q_1$, its image $q_1$ under $\varphi$ is a singular point of the cubic surface $S_1$ which lies on the plane $H_2$. It follows that $q_1$ has multiplicity 3 in $H_2+S_1.$ The intersection of $H_2\cap S_1$ is also a union of a line and a smooth conic, which is the image under $\varphi$ of the twisted cubic $R$ and the the line $l$ respectively. Therefore the configurations of the divisors $H_2+S_1$ and $S_2+H_3$ are identical. Applying the same construction of $\psi$, we obtain a similar crepant birational map 
$$
\psi':(\bP^3, H_2+S_1)\dashrightarrow (\bP^3,S_4+H_5)
$$
with 
\begin{align}\label{eqn:trivcpsiellip}
    \mathbf c(\psi')=0.
\end{align}
As above, $H_5\cap S_4$ is a union of a line and a conic intersecting at two $\Bbbk$-rational points, $S_4$ is a cone over this reducible curve and $H_5$ is a plane.

\subsection{Constructing the maps $\eta$ and $\eta'$} 
By \cite{Du24} (in particular the proof of Theorem 1.2 and Example 1.3 there), there exists a crepant birational map:
$$
\eta_1\colon (\bP^3,S_3+H_4)\dashrightarrow (\bP^1\times\bP^2,\Delta'),
$$
where $(\bP^1\times\bP^2,\Delta')$ is the toric model with
$$
\Delta'=\left(\{0\}\times\bP^2\right)+\left(\{0\}\times\bP^2\right)+\left(\bP^1\times E\right)
$$
and $E$ is a union of three lines forming a triangle in $\bP^2$. Similarly as above, one can check that $\eta_1$ is defined over $\Bbbk$. On the other hand, since $S_4+H_5$ is of the same configuration, there exists a similar crepant birational map over $\Bbbk$
$$
\eta_1'\colon (\bP^3,S_4+H_5)\dashrightarrow (\bP^1\times\bP^2,\Delta')
$$
such that 
\begin{align*}
    \mathbf{c}(\eta_1^{-1})+\mathbf{c}(\eta_1')=0.
\end{align*}
Consider the crepant birational map 
$$
\eta_2\colon(\bP^1_{u_1,u_2}\times\bP^2_{t_1,t_2,t_3},\Delta')\dashrightarrow (\bP^3,\Delta_3)
$$
$$
(u_1:u_2)\times(t_1:t_2:t_3)\mapsto(1:\frac{u_2}{u_1}:\frac{t_2}{t_1}:\frac{t_3}{t_1})
$$
where $(\bP^3,\Delta_3)$ is the standard toric pair. Then put 
$$
\eta=\eta_2\circ\eta_1, \qquad \eta'=\eta_2\circ\eta_1',
$$
we obtain the desired maps 
$$
\eta\colon (\bP^3,S_3+H_4)\dashrightarrow(\bP^3,\Delta_3),\qquad\eta'\colon (\bP^3,S_4+H_5)\dashrightarrow(\bP^3,\Delta_3)
$$
with 
\begin{align}\label{eqn:trivcelleta}
    \bc(\eta^{-1})+\bc(\eta')=\mathbf{c}(\eta_1^{-1})+\mathbf{c}(\eta_1')=0.
\end{align}

\subsection{Proof of Proposition~\ref{prop:ellip}}
Put 
$$
\sigma:=\eta'\circ\psi'\circ\varphi\circ\pi^{-1}\circ\psi^{-1}\circ\eta^{-1}.
$$
Then $\sigma$ is a crepant birational automorphism over $\Bbbk$ of the standard toric pair $(\bP^3,\Delta_3)$. Combining \eqref{eqn:trivcellippi}, \eqref{eqn:trivcpsi'ellip}, \eqref{eqn:trivcpsiellip} and \eqref{eqn:trivcelleta}, we conclude that 
$$
0\ne\bc(\sigma)=\bc(\varphi)=(C\times\bP^1,0)-(C'\times\bP^1,0)\in \mathbf{DivBurn}_2(\Bbbk).
$$

\begin{coro}\label{coro:main3}
    Under the same assumption on the field $\Bbbk$ as in Proposition~\ref{prop:ellip}, there exists a volume preserving birational automorphism $\sigma'$ over $\Bbbk$ of the pair $[\bP^3,\omega_3]$, where $\omega_3$ is the standard torus invariant volume form, such that 
    $$
   0\ne\mathbf{c}(\sigma')\in\Burnf_2(\Bbbk).
    $$
    In particular, $\Bir_\Bbbk(\bP^n,\omega_n)$ is not generated by pseudo-regularizable elements for $n\geq 3$.
\end{coro}
\begin{proof}
    Let $\sigma$ be the crepant birational automorphism of the standard toric pair with $\bc(\sigma)\ne 0$ constructed in Proposition~\ref{prop:ellip}. By Corollary~\ref{coro:volumepreserve}, there exists a positive integer $N$ such that $\sigma^N$ is a volume preserving birational automorphism of $[\bP^3,\omega_3]$. Put $\sigma'=\sigma^N$. Remark~\ref{rema:samec} shows that
    $$
    0\ne\bc(\sigma')=N\cdot([C\times\bP^1,0]-[C'\times\bP^1,0])\in\Burnf_2(\Bbbk).
    $$
The map $\sigma'$ naturally extends to a volume preserving birational automorphism of the pair $[\bP^3\times\bP^r,\omega_3\wedge\omega_r]
$ over $\Bbbk$ with the identity map on $\bP^r$ for $r\geq 0$, and thus also to a volume preserving birational automorphism $\tilde\sigma$ of the pair 
$
[\bP^{3+r}_{\Bbbk},\omega_{3+r}].
$
Since $C$ and $C'$ are not stably birational over $\Bbbk$, we know that
$$
    0\ne\bc(\tilde\sigma)=N\cdot([C\times\bP^{r+1},0]-[C'\times\bP^{r+1},0])\in\Burnf_{r+2}(\Bbbk).
    $$
        The last assertion then follows from Example \ref{exam:vanishpseudo}.
\end{proof}

\begin{coro}\label{coro:nonsimple3}
     Under the same assumption on the field $\Bbbk$ as in Proposition~\ref{prop:ellip}, and for $n\geq 3$, there exists a surjective homomorphism  
     $$
     \Bir_\Bbbk(\bP^n,\omega_n)\to A
     $$
     where $A\simeq\bigoplus_J \bZ$ and $J$ is a set of the same cardinality as $\Bbbk.$ In particular, $\Bir_\Bbbk(\bP^n,\omega_n)$ is not simple when $n\geq 3$.
\end{coro}
\begin{proof}
    Let $I$ be the set of isomorphism classes of quintic genus one curves $C$ with no $\Bbbk$-points. By \cite[Lemma 3.8]{SL}, the set $I$ has the cardinality of $\Bbbk.$ Let $J$ be the set of unordered pairs 
    $$
    J:=\{(C,\mathrm{Jac}^2(C)): C\in I\}.
    $$
    Then $J$ also has the cardinality of $\Bbbk$. Put 
    $$
    A':=\bigoplus_{(C,\mathrm{Jac}^2(C))\in J}\bZ \cdot \left([C\times\bP^{n-2},0]-[\mathrm{Jac}^2(C)\times\bP^{n-2},0]\right)\subset\Burnf_{n-1}(\Bbbk).
    $$
    Since distinct elements in $I$ are not stably birational to each other, we have $A'\simeq \bigoplus_J\bZ$. Consider the projection map 
    $$
    \mathrm{pr}:\Burnf_{n-1}(\Bbbk)\to A'.
    $$ 
    Then the image of the map
    $$
    \mathrm{pr}\circ \bc: \Bir_\Bbbk(\bP^n,\omega_n)\to A'
    $$
    is a free abelian group $A$ also isomorphic to $\bigoplus_J\bZ.$ 
\end{proof}

\begin{rema}
  The study of varieties over nonclosed fields is closely related to that of finite group actions on varieties over algebraically closed fields. In particular, one can find an action of the
    cyclic group $C_5$ on a quintic elliptic curve $C$ with no $G$-fixed points, where $C$ is realized as a linear section of the Grassmannian $\Gr(2,5)$, see e.g., \cite[Example 10]{KTorbit}, \cite[Section 6]{CTZcub}. There is a $C_5$-equivariant diagram similar to that in \eqref{dia:ellip}: all pertinent geometric objects appearing in the construction can be chosen to be $C_5$-equivariant. This extends the construction of a nontrivial equivariant birational invariant $c_G$, introduced in \cite[Section 7]{KTorbit}, to the context of log Calabi-Yau pairs.  
\end{rema}
\

\section{$\bP^4$ with K3 surfaces}\label{sect:K3}
Throughout this section, we work over $\Bbbk=\bC$. The goal is to construct a crepant birational automorphism $\sigma$ of the standard toric pair $(\bP^4, \Delta_4)$ such that $\mathbf{c}(\sigma)\ne 0$.
The main ingredient is the Cremona transformation of $\bP^4$ studied in \cite{HL}. First, we recall their construction, see also \cite[Theorem 3.12]{SL}.

\begin{theo}[{\cite[Theorem 2.1, Theorem 3.1]{HL}}]\label{theo:HLmain}
Let $(R_L,\Gamma)$ be a general polarized K3 surface with a polarization $\Gamma$ such that $\Gamma^2 = 12$ and let $\{ x_1, x_2, x_3 \} \subset R_L$ be a general triple of points. 
The linear system $|\Gamma|$ defines an embedding $R_L \to \bP^7$. Let $S_L \subset \bP^4$ be the proper transform of $R_L$ under the projection $\bP^7 \dashrightarrow \bP^4$ from the plane generated by $\{ x_1, x_2, x_3 \}$. Then there exists the following commutative diagram:
\begin{equation}\label{eqn:HLcremona}
\begin{tikzcd}
& X \ar[rd, "\pi_M"] \ar[dl, swap, "\pi_L"] & \ \\
S_L\subset \bP^4 \ar[rr, dashed, "\psi"] & & \bP^4 \supset S_M
\end{tikzcd}
\end{equation}
where
\begin{enumerate}
\item 
$\psi$ is a birational map given by the linear system $\mathcal{L}_1$ of quartics containing $S_L.$ The base locus $\mathrm{Bs}(\mathcal{L}_1)=S_L$;
\item
$\psi^{-1}$ is also given by a linear system $\mathcal{L}_2$  of quartics, where $\mathrm{Bs}(\mathcal{L}_2)$ is a surface $S_M$. Similarly, there exists a K3
surface $R_M$ of degree $12$ and three points $\{ x'_1, x'_2, x'_3 \}$ on $R_M$ such that  $S_M\subset\bP^4$ is the proper transform of $R_M$ under the projection $\bP^7 \dashrightarrow \bP^4$ from the plane generated by $\{ x_1', x_2', x_3' \}$;
\item 
$S_L$ and $S_M$ are non-normal surfaces of degree $9$. The singular locus of each surface consists of three transverse double points, denoted by $\Sing(S_L)=\{p_1, p_2, p_3\}$ and $\Sing(S_M)=\{q_1, q_2, q_3\}$;

\item 
$\pi_L$ is the blow up of $S_L$, $\pi_M$ is the blow up of $S_M$.
\end{enumerate}
Furthermore, if $\mathrm{Pic}(R_L) = \mathbb{Z} \cdot \Gamma$, then $R_M$ is the unique Fourier–Mukai partner of $R_L$ which is not isomorphic to $R_L$. 
\end{theo}

\begin{coro}
Let $\psi$ be the Cremona transformation constructed in Theorem~\ref{theo:HLmain} from a very general K3 surface of degree $12$. One has
\begin{align}\label{eqn:slk3}
    0\ne c(\psi)=[S_L\times\bP^1]-[S_M\times\bP^1]\in\Burn_3(\bC).
\end{align}
\end{coro}
\begin{proof}
By Torelli theorem, the Picard rank of a very general projective K3 surface of degree $12$ over $\bC$ is $1$. By Theorem~\ref{theo:HLmain}, we know that $S_L$ is not isomorphic to $S_M$. It follows that $S_L$ and $S_M$ are not stably birational, since stably birational K3 surfaces are isomorphic. From diagram~\eqref{eqn:HLcremona}, we see that $c(\psi)=[S_L\times\bP^1]-[S_M\times\bP^1]\neq 0$.
\end{proof}
 

\begin{rema}
\label{rem-points-span-plane}
When $\Pic(R_L)=\bZ\cdot \Gamma$, we call the map $\psi$ in Theorem~\ref{theo:HLmain} an {\em HL-Cremona transformation.} We refer to an HL-Cremona transformation constructed from a very general K3 surface of degree $12$ as a {\em very general HL-Cremona transformation}. For a very general HL-Cremona transformation, we may assume that the span of the singular points $p_1,p_2,p_3$ of the surface $S_L$ (resp., $q_1,q_2,q_3$ of the surface $S_M$) is a plane in $\bP^4$, denoted by $\Pi_L$ (resp., $\Pi_M$). In addition, we may also assume that the $\Pi_L$ intersects $S_L$ (resp., $\Pi_M$ intersects $S_M$) in a zero-dimensional set. 
To justify these assumptions, one can check that they are open conditions which hold in the explicit example constructed in \cite[Section 2]{HL}. Therefore they are also true for a very general HL-Cremona transformation.
In what follows, we work with a very general HL-Cremona transformation.
\end{rema}

\begin{prop}\label{prop:K3}
Let $\psi$ be a very general HL-Cremona transformation. 
Then there exist hyperplanes $H_L$ and $H_M$, and quartic hypersurfaces $B_L$ and $B_M$ in $\bP^4$ such that $\psi$ extends to a crepant birational map between log Calabi-Yau pairs 
$$
\psi: (\bP^4,H_L+B_L)\dashrightarrow(\bP^4,H_M+B_M).
$$
Moreover, there exist crepant birational maps $\varphi,\varphi',\eta,\eta'$ and 
$$
\sigma:=\eta'\circ\varphi'\circ\psi\circ\varphi^{-1}\circ\eta^{-1}
$$
with the following commutative diagram
\begin{equation}\label{eqn:K3dia}
    \begin{tikzcd}
(S_L\subset \bP^4, H_L+B_L)\ar[rr, "\psi", dashed]\ar[dd,"\varphi",dashed]&  &(S_M\subset\bP^4, H_M+B_M) \ar[dd ,"\varphi'",dashed] \\
\\
(\bP^4,\Xi_L+\Xi_L'+Z_L)\ar[dd,"\eta", dashed]&&(\bP^4,\Xi_M+\Xi_M'+Z_M)\ar[dd, "\eta'", dashed]\\
\\
(\bP^4,\Delta_4)\ar[rr, "\sigma",dashed]&&(\bP^4,\Delta_4)
\end{tikzcd}
\end{equation}
where 
\begin{itemize}
    \item $ \Xi_L, \Xi_L', \Xi_M, \Xi_M'$ are hyperplanes in $\bP^4$,
    \item $Z_L$ and $Z_M$ are cones over smooth cubic surfaces,
    \item $(\bP^4,\Delta_4)$ is the standard toric pair.
\end{itemize}
In particular, $\sigma$ is a crepant birational automorphism of the standard  4-dimensional toric pair $(\bP^4,\Delta_4)$.
\end{prop}

\begin{prop}\label{prop:K3cmap}
    Let $\sigma$ be the map as in Proposition~\ref{prop:K3}. Then we have 
$$
0\ne\mathbf{c}(\sigma)\in\dBurn_3(\bC).
$$
\end{prop}

The rest of the section is devoted to a constructive proof of Proposition~\ref{prop:K3}, with the construction of each map in the diagram \eqref{eqn:K3dia} explained in different subsections. We begin with a detailed analysis of the geometry of HL-Cremona transformations. 

\subsection{Geometry of HL-Cremona transformations}\label{subsect:HLcre}
Here we recall some facts from \cite{HL}. First, there is a factorization of $\pi_L$ into a sequence of blowups in smooth centers, as in the following diagram extending  \eqref{eqn:HLcremona}:

\begin{equation}\label{eqn:HLcremonafact}
\begin{tikzcd}
&P'\ar[rd,"\xi_L"]\ar[ld, swap, "\beta_L"]&\\
P\ar[rd, "\alpha_L"]&& X \ar[rd, "\pi_M"] \ar[dl, swap, "\pi_L"]   \\
&S_L\subset \bP^4 \ar[rr, dashed, "\psi"] & & \bP^4 \supset S_M
\end{tikzcd}
\end{equation}
where 
\begin{itemize}
    \item $P$ is the blowup of $\bP^4$ at three singular points $p_1, p_2, p_3$ of $S_L$. Let $E_1, E_2, E_3$ be the corresponding exceptional divisors, and $S_L'$ be the strict transform of $S_L$ in $P$. Then $E_i\simeq\bP^3$ and $S_L'$ intersects $E_i$ in two skew lines, for $i=1,2,3$;
    \item $P'$ is the blowup of $P$ along $S_L'$. Denote by $E$ the exceptional divisor of $\beta_L$. We denote the strict transform of $E_1, E_2$ and $E_3$ by $E_1'$, $E_2'$ and $E_3'$, respectively. Then $E_i'$ is isomorphic to $\bP^3$ blown up along two skew lines, and thus $E_i'$ has the structure of a $\bP^1$-bundle over $\bP^1\times\bP^1$, for $i=1,2,3$;
    \item The map $\xi_L\colon P'\to X$ is a blowdown of each $E_i'$ to a   surface $Q_i\subset X$ where $Q_i\simeq \bP^1\times\bP^1$. Let $E_X$ be the image of $E$ on $X$. By construction, each $Q_i$ is contained in $E_X$. Note that $\pi_L(Q_i)=p_i$ and $E_X$ is the exceptional divisor of $\pi_L$. Let $K_i=\pi_M^{-1}(q_i), i=1,2,3.$ By symmetry, each $K_i$ is also isomorphic to $\bP^1\times\bP^1$. 
\end{itemize}

We also have the following diagram involving the K3 surface $R_L$
\begin{equation}\label{eqn:K3-surface-diagram}
\begin{tikzcd}
S_L' \ar[rr, "\kappa"] \ar[drr, swap, "\alpha_L"] &  & \Sigma_L \ar[d, "\nu"] \ar[d, rr, "\mu"]\ & & R_L \ar[dll, dashed, "\pi"]\\
 & & S_L & &
\end{tikzcd}
\end{equation}
where 
\begin{itemize}
\item $\pi$ is a projection from 3 general points on $R_L$;
\item 
$\mu$ is the blow up of these 3 points, i.e., the base locus of $\pi$;
\item 
$\nu$ is the normalization which maps some points $p_{i,j}$ to the transverse double points $p_i$, for $j=1,2$, $i=1,2,3$;
\item 
$\kappa$ is the blow up of the points $p_{i,j}$ for $i=1,2,3$, $j=1,2$;
\item 
$\alpha_L$ is as in the diagram \eqref{eqn:HLcremonafact}.
\end{itemize}

A diagram similar to \eqref{eqn:K3-surface-diagram} exists for the surface $S_M$, so in what follows we will freely use the notation $ S_M'$, $\Sigma_M$, $R_M$. 

Here we introduce more notation. Let $\Gamma_L$ be the image under $\nu$ in $S_L$ of the polarization of $R_L$, and $\tilde \Gamma_L$ its preimage in $X$ via the map $\pi_L$. Similarly, $\Gamma_M$ and $\tilde \Gamma_M$ refer to the corresponding objects coming from $R_M$. Denote by $F_1, F_2$ and $F_3$ the images in $S_L$ of the $\mu$-exceptional curves  under $\nu$, and their preimages in $X$ by $\tilde F_1, \tilde F_2, \tilde F_3$. Note that $F_1, F_2$ and $F_3$ are $(-1)$-curves on $S_L.$
By symmetry, there are also three $(-1)$-curves on $S_M$. We refer to them as $ G_1,G_2, G_3$, and to their preimages in $X$ as $\tilde G_1, \tilde G_2, \tilde G_3$. 
We also denote by $L$ (resp., $M$) the class of a general hyperplane section on $\bP^4$ containing $S_L$ (resp., $S_M$). Let $L_X$ (resp., $M_X$) be the pullback of the class $L$ (resp., $M$) to $X$.

For clarity, we collect notation that we will use for the remainder of this section.
\begin{itemize}
    \item $L$, $M$: general hyperplane sections on $\bP^4$; 
    \item $F_1, F_2, F_3$: $(-1)$-curves on $S_L$;
    \item $G_1, G_2, G_3$: $(-1)$-curves on $S_M$;
      \item $\Gamma_L$, $\Gamma_M$: images of polarizations of $R_L$ and $R_M$ in $S_L$ and $S_M$;
    \item $ \tilde \Gamma_L,\tilde\Gamma_M,\tilde F_i, \tilde G_i$: preimages of $\Gamma_L, \Gamma_M, F_i, G_i$ in $X$;
      \item $Q_1, Q_2, Q_3$: quadric surfaces in $X$ above $p_1, p_2, p_3$;
    \item $K_1, K_2, K_3$: quadric surfaces in $X$ above $q_1, q_2, q_3$;
    \item $E_1, E_2, E_3$: exceptional divisors in $P$ above $p_1, p_2, p_3$;
    \item $S_L'$: the strict transform of $S_L$ in $P$;
    \item $E_1', E_2', E_3'$: strict transforms of $E_1, E_2, E_3$ in $P'$;
    \item $L_X, M_X, L', M'$: pullback of $L$ and $M$ in $X$ and $P'$ respectively;
    \item $E$: the exceptional divisor in $P'$ above $S_L'$.
\end{itemize}

\begin{lemm}[{\cite[Section 3]{HL}}]
\label{lem-K3-intersection-theory}
    Let
    $
    \rH^4(X,\bZ)_{\mathrm{alg}}
    $ be the lattice spanned by the algebraic classes in the middle cohomology $\rH^4(X,\bZ)$. Then $
    \rH^4(X,\bZ)_{\mathrm{alg}}
    $  is spanned by the classes 
    $$
    L_X^2, \tilde \Gamma_L, \tilde F_1,\tilde F_2,\tilde F_3, Q_1, Q_2, Q_3
    $$
    with the entries of the intersection matrix given by
    \begin{align*}
       L_X^4=1,\quad \tilde \Gamma_L^2=-12, \quad\tilde F_i^2=1, \quad Q_i^2=1, \,\,i=1,2,3\quad \text{ and }\, 0 \text{ in other entries}.
    \end{align*}
   By symmetry,  $
    \rH^4(X,\bZ)_{\mathrm{alg}}
    $ is also spanned by the classes coming from $S_M$:
    $$
    M_X^2, \tilde \Gamma_M, \tilde G_1,\tilde G_2,\tilde G_3, K_1, K_2, K_3
    $$
    with a similar intersection matrix. The two sets of generators are related by
    $$
    \begin{pmatrix}
        M_X^2\\\tilde\Gamma_M\\\tilde G_1\\\tilde G_2\\\tilde G_3\\K_1\\K_2\\K_3
    \end{pmatrix}=
    \begin{pmatrix}
    7&-3&4&4&4&2&2&2\\
    36&-17&24&24&24&12&12&12\\
    4&-2&3&3&3&2&1&1\\
    4&-2&3&3&3&1&2&1\\    
    4&-2&3&3&3&1&1&2\\
    2&-1&2&1&1&1&1&1\\
    2&-1&1&2&1&1&1&1\\
    2&-1&1&1&2&1&1&1\\
   \end{pmatrix}
   \begin{pmatrix}
       L_X^2\\\tilde\Gamma_L\\\tilde F_1\\\tilde F_2\\\tilde F_3\\Q_1\\Q_2\\Q_3
   \end{pmatrix}
    $$
\end{lemm}

\begin{lemm}[{\cite[Corollary 2.2]{HL}}]
\label{lem-intersection-theory}
Intersection theory on $P'$ is as follows:
\[
L' E'_i=0, \quad E^3 E'_i=-4, \quad E^2(E'_i)^2=2, \quad E (E'_i)^3=0, \quad (E'_i)^4=-1, \quad (L')^3E =0.
\]

\end{lemm}


The following lemmas are crucial to our constructions.

\begin{lemm}
\label{lem-plane-class}
    Let $\Pi_L$ be the plane in $\bP^4$ spanned by the singular points $p_i$ of $S_L$, with $1\leq i\leq 3$. Then the class $\tilde{\Pi}_L\in\rH^4(X,\bZ)_{\mathrm{alg}}$ of the strict transform of $\Pi_L$ in $X$ under $\pi_L$ equals to
    $$
    \tilde{\Pi}_L=L_X^2-\sum_{i=1}^3Q_i.
    $$
\end{lemm}
\begin{proof}
By Remark \ref{rem-points-span-plane}, we know that $p_1,p_2,p_3$ indeed span a plane in $\bP^4$. 
By \cite[\href{https://stacks.math.columbia.edu/tag/0B0I}{Section 0B0I}]{stacks-project}, the pullback $\pi_L^*(\Pi_L)$ of the class $\Pi_L$ is equal to the class of the scheme-theoretic preimage $\pi_L^{-1}(\Pi_L)$ (where we consider only components of codimension $2$). Since $\Pi_L$ intersects $S_L$ in a zero-dimensional set containing $p_1,p_2,p_3$ by Remark \ref{rem-points-span-plane}, going along the left side of the diagram \eqref{eqn:HLcremonafact}, we find
\[
\tilde{\Pi}_L + \sum_{i=1}^3 Q_i = \pi_L^*(\Pi_L) = \pi_L^*(L^2) = (\pi_L^*(L))^2 = L_X^2. 
\]
\end{proof}

\begin{lemm}\label{lemm:piplane}
 Let $\Pi_L$ be the plane spanned by the singular points $p_1,p_2,p_3$ of $S_L$, and $\Pi_M$ be the plane spanned by the singular points $q_1,q_2,q_3$ of $S_M$. Then the strict transform $\psi_* \Pi_L$ is equal to $\Pi_M$. 
\end{lemm}
\begin{proof}
Let $\tilde\Pi_L$ be the strict transform of ${\Pi}_L$ on $X$. By Lemma \ref{lem-plane-class}, its class in $\rH^4(X,\bZ)_{\mathrm{alg}}$ 
equals to $L_X^2 - \sum_{i=1}^3 Q_i$. Using Lemma~\ref{lem-K3-intersection-theory}, one can
compute 
$$
M_X^2(L_X^2 - \sum_{i=1}^{3} Q_i) = 1.
$$ 
This implies that $\pi_M$ maps $\tilde \Pi_L$ to a plane in $\bP^4$. Lemma \ref{lem-K3-intersection-theory} also shows that
\[
L_X^2 - \sum_{i=1}^{3} Q_i = M_X^2 - \sum_{i=1}^{3} K_i. 
\]
Intersecting this class with $K_i$, we find that the intersection number is nonzero for each $i=1,2,3$, and thus the plane $\psi_* \Pi_L$ contains the singular points $q_1,q_2,q_3$. It follows that $\psi_*\Pi_L=\Pi_M$.   
\end{proof}

\begin{lemm}\label{lemm:quadint}
The image $\pi_M(Q_i)$ is a smooth quadric surface containing singular points $q_j$ of $S_M$ for $i,j=1,2,3$.
Similarly, the image $\pi_L(K_i)$ is a smooth quadric surface containing singular points $p_j$ of $S_L$ for $i,j=1,2,3$.  
\end{lemm}
\begin{proof}
By symmetry, it suffices to prove the first assertion. By Lemma \ref{lem-K3-intersection-theory}, we have $Q_i M^2 = 2$, so $M|_{Q_i}$ is a linear system of degree $2$ on $Q_i=\bP^1\times\bP^1$ that defines a morphism $(\pi_M)|_{Q_i}$. Thus, $(\pi_M)|_{Q_i}$ is an isomorphism and 
$\pi_M(Q_i)$ is a smooth quadric surface in $\bP^3\subset \bP^4$. 

Next, we prove that $\pi_M(Q_i)$ contains the points $q_j$ for $j=1,2,3$. It suffices to show that the intersection $Q_i\cap K_j$ is non-empty for any $i, j=1,2,3$. Indeed, using Lemma \ref{lem-K3-intersection-theory}, we compute the intersection numbers:
\[
Q_i K_j = Q_i \left(2L_X^2-\tilde \Gamma_L+F_j+\sum_{k=1}^3(\tilde F_k+Q_k)\right) = 1,\quad i,j=1,2,3.
\]
\end{proof}

\subsection{Constructing the boundaries}\label{sect:4bound}
We retain the notation above. Let 
$$
\psi:\bP^4\dashrightarrow\bP^4
$$
be a very general HL-Cremona transformation. We seek divisors $D_L$ and $D_M$ such that 
\begin{itemize}
    \item  $(\bP^4,D_L)$ and $(\bP^4, D_M)$ are log CY pairs admitting toric models, and
    \item $\psi$ extends to a crepant birational map of pairs 
$
(\bP^4,D_L)\dashrightarrow(\bP^4, D_M).
$
\end{itemize}
To achieve this, we construct $D_L$ and $D_M$ as follows:
\begin{enumerate}
    \item Pick the singular point $p_1$ of $S_L$. Consider the quadric surface $Q_1$ in $X$ above $p_1$. Lemma \ref{lemm:quadint} shows that $\pi_M(Q_1)$ is a smooth quadric surface in $\bP^4$. Let $H_M=\bP^3$ be the unique hyperplane in $\bP^4$ containing $\pi_M(Q_1)$, and $B_L$ the strict transform $\psi^{-1}_{*}(H_M)$. Note that $B_L$ is a quartic hypersurface in $\bP^4$ containing~ $S_L$;
    \item We proceed by symmetry: pick the singular point $q_1$ of $S_M$. Consider the quadric surface $K_1$ above $q_1$. Its image $\pi_L(K_1)$ is a quadric surface in $\bP^4$. Let $H_L=\bP^3$ be the unique hyperplane containing $\pi_L(K_1)$, and $B_M$ the strict transform $\psi_{*}(H_L)$. Similarly, $B_M$ is a quartic hypersurface containing $S_M;$
    \item Put $D_L=H_L+B_L$ and $D_M=H_M+B_M$. By construction, one has the inclusions
    \begin{equation}
    \label{eq-K3-Q1-H2}
            \pi_M(Q_1)\subset H_M,\quad\pi_L(K_1)\subset H_L,\quad S_L\subset B_L\quad \text{and}\quad S_M\subset B_M.
    \end{equation}
\end{enumerate}

By symmetry, the same construction starting from other singular points of $S_L$ and $S_M$ will produce the same results. Without loss of generality, we work with the choice of $p_1$ and $q_1$.

 By construction, $S_L$ is contained in $B_L$ but not in $H_L$, and $S_M$ is contained in $B_M$ but not in $H_M$. It follows that $\psi$ extends to a crepant birational map of the CY pairs
$$
\psi:(\bP^4,H_L+B_L)\dashrightarrow (\bP^4,H_M+B_M).
$$
In the following subsections, we construct crepant birational maps from these two pairs to the standard toric pairs.

\subsection{Computing the strata}
Here we compute the strata $B_L\cap H_L$ and $B_M\cap H_M$ from our construction above.
\begin{lemm}\label{lemm:HLBL}
We have the inclusions
$
\Pi_L \subset B_L $, and $\,\Pi_M \subset B_M.
$
\end{lemm}
\begin{proof}
First, recall from \eqref{eq-K3-Q1-H2} that $\pi_M(Q_1)\subset H_M$. On the other hand, it follows from Lemma \ref{lemm:quadint} that $q_i\in \pi_M(Q_1)$ for $i=1,2,3$. This implies that $\Pi_M \subset H_M$ and $\psi^{-1}_*(\Pi_M)\subset \psi^{-1}_*(H_M)$. By Lemma~\ref{lemm:piplane}, we know $\psi^{-1}_*(\Pi_M)=\Pi_L$. By construction, we have $\psi^{-1}_*(H_M)=B_L.$ This shows $\Pi_L\subset B_L$. By symmetry, $\Pi_M\subset B_M$ also holds.
\end{proof}

\begin{lemm}\label{lemm:piLKiBL}
Let $H\subset \bP^4$ be any hyperplane section such that $p_i\in H$ for all $i=1,2,3$. Let $H_X$ be the strict transform of $H$ on $X$. Then $Q_i\subset H_X$ for $i=1,2,3$. In particular, this implies that
\[
\pi_M(Q_i)\subset B_M\quad \text{and} \quad\pi_L(K_i)\subset B_L\quad  \text{for }\quad i=1,2,3.
\]
\end{lemm}
\begin{proof}
Let $H_P$ be the strict transform of $H$ on $P$ as in diagram \eqref{eqn:HLcremonafact}. Fix $i\in \{1,2,3\}$. Observe that the intersection $\Psi=H_P\cap E_i$ is a plane in $E_i\simeq \bP^3$, where $E_i$ is the exceptional divisor over $p_i$. Note that $E_i \cap S'_L$ is a union of two skew lines $l_{i,1}$, $l_{i,2}$ in $E_i\simeq \bP^3$ where $S'_L$ is the strict transform of $S_L\subset \bP^4$ on $P$. Two cases are possible: either both skew lines $l_{i,1}$, $l_{i,2}$ intersect the plane $\Psi$ transversally, or one of the skew lines is contained in $\Psi$. 

Start with the first case. After the blow up $\beta_L\colon P'\to P$, the strict transform $\Psi'$ of $\Psi$ is a del Pezzo surface of degree $7$ which surjects to $Q_i$ under the map $\xi_L\colon P\to X$. Thus, the strict transform $H_X$ of $H$ on $X$ contains $Q_i$, as is claimed. 

Then we show the second case is impossible. Without loss of generality, assume $l_{i,1}\subset \Psi$. This implies that the hyperplane $H$ contains one of the two branches of the surface $S_L$ near the point $p_i$, which is a contradiction since $S_L$ is not contained in a hyperplane, and $H$ intersects $S_L$ in a curve. 
\end{proof}

\begin{prop}\label{prop:mult3}
    The point $p_1$ has multiplicity 3 in $B_L$. The point $q_1$ has multiplicity 3 in $B_M$.
\end{prop}
\begin{proof}
By symmetry of the construction, it suffices to prove this claim for the point $p_1$ on $B_L$. Let $d$ be the multiplicity of $p_1$ in $B_L$. Since $B_L$ is a quartic, we have $d\leq 4$. We exclude the cases $d=1, 2$ and $4$. 

We keep the notation in diagram~\eqref{eqn:HLcremonafact}. Let $B_L'$ be the strict transform of $B_L$ on $P$. Note that since $S_L\subset B_L$, we have $S_L'\subset  B_L'$. 
Then the intersection of $B_L'$ with $E_1=\bP^3$ is a degree $d$ surface in $E_1$ containing two skew lines $S'_L\cap E_1$.  This implies $d\geq 2$. 

Assume $d=2$. Then $N_1 = B_L'\cap E_1$ is a quadric surface containing these skew lines. Hence $N_1$ is either a smooth quadric, or a union of two planes. Indeed, the cases of a quadric cone and a double plane are excluded since these surfaces in $\bP^3$ cannot contain a pair of skew lines. 

Assume that $N_1$ is a smooth quadric surface $\bP^1\times \bP^1$. Note that, in this case, $ B_L'$ is smooth near $N_1$. One checks that the image of $N_1$ under the map (induced from \ref{eqn:HLcremonafact})
$$
(\xi_L\circ\beta_L^{-1})\mid_{E_1}\colon E_1\to \bP^1\times \bP^1
$$
is $\bP^1$. This implies that the strict transform of $B_L$ in $X$, which equals to the strict transform of $H_M$ on $X$, does not contain the quadric surface $Q_1$. This contradicts to \eqref{eq-K3-Q1-H2} stating that $\pi_M(Q_1)\subset H_M$. So $N_1$ is not a smooth quadric surface.


Assume now that $N_1$ is a union of two planes containing two skew lines $S_L'\cap \bP^3$. Note that $B_L'$ is smooth at the generic points of these lines. As above, we show that  
$$
(\xi_L\circ\beta_L^{-1})\mid_{E_1}\left( N_1\right)\ne\bP^1\times\bP^1.
$$
Indeed, the intersection of the strict transform of $B_L'$ in $P'$ and $E_1'$ consists of the strict transform $N'_1$ of $N_1$ on $P'$, which is a union of two Hirzebruch surfaces $\mathbb{F}_1$, and two exceptional divisors of the map $E'_1\to E_1$. Neither of these $4$ components maps surjectively to $\bP^1\times\bP^1$ under the map $\xi_L\mid_{E_1'}: E'_1\to \bP^1\times\bP^1$. This implies that the strict transform of $B_L$ in $X$ does not contain the quadric surface $Q_1$, which is a contradiction as is explained above. It follows that $d\geq 3$.

Assume that $d=4$. Then $B_L$ is a cone over a surface of degree $4$, and $p_1$ is its vertex. By Lemma \ref{lemm:HLBL}, we have that $p_1\in H_L$.  Then $B_L\cap H_L$ is a cone over some curve.  By Lemma \ref{lemm:piLKiBL} and the construction of $H_L$, we know that $B_L\cap H_L$  contains the smooth quadric surface $\pi_L(K_1)$. 
However, a smooth quadric is not a cone, which is a contradiction. It follows that $d\ne 4.$
\end{proof}

\begin{rema}\label{rema:smoothcubicgeneral}
   Proposition~\ref{prop:mult3} implies that the exceptional divisor over $p_1$ in $\Bl_{p_1}B_L$ is a cubic surface. For a very general HL-Cremona transformation, we may assume the cubic surface is smooth. Indeed, one can check that this is an open condition and holds for the specific example given in \cite{HL}.
\end{rema}

\begin{prop}\label{prop:HLBLint}
For a very general HL-Cremona transformation, the boundary divisors constructed in Section~\ref{sect:4bound} satisfy
\begin{align*}
    B_L \cap H_L=\ & \pi_L(K_1)\ \cup\ \Pi_L\ \cup\ \Lambda_L \\ 
    =\ & \bP^1\times\bP^1\ \cup\ \bP^2\ \cup\ \bP^2,
\end{align*}
where $K_1$ is the quadric surface  over the point $q_1$; $\Pi_L$ is the plane spanned by $p_1,p_2,p_3$; and $\Lambda_L$ is a plane different from $\Pi_L$. 
By symmetry, we also have
\begin{align*}
    B_M \cap H_M =\ & \pi_M(Q_1)\ \cup\ \Pi_M\ \cup\ \Lambda_M \\ 
    =\ & \bP^1\times\bP^1\ \cup\ \bP^2\ \cup\ \bP^2
\end{align*}
where $Q_1$ is the quadric surface over the point $p_1$; $\Pi_M$ is the plane spanned by $q_1,q_2,q_3$; and $\Lambda_M$ is a plane different from $\Pi_L$.
\end{prop}
\begin{proof}
It suffices to show the first assertion. First, recall that by construction,  we have $
H_L\supset\pi_{L}(K_1).
$ Lemma~\ref{lemm:quadint} implies that $H_L\supset \Pi_L$. Lemma~\ref{lemm:HLBL} and \ref{lemm:piLKiBL} show that $B_L\supset \Pi_L$ and $B_L\supset \pi_L(K_1)$ respectively. 
Therefore $B_L\cap H_L$ contains $\pi_L(K_1)\cup\Pi_L$. Since $B_L\cap H_L$ is a degree four surface, there is a residual component $\Lambda_L$ isomorphic to $\bP^2$ in the intersection. 

It remains to check that $\Lambda_L$ is different from $\Pi_L$, or equivalently, $B_L\cap H_L$ is reduced. 
Recall from Section~\ref{sect:4bound} that once we make a choice of points $p_1\in S_L$ and $q_1\in S_M$, the boundary $B_L+H_L$ (and thus the intersection $B_L\cap H_L$) is canonically defined. 
It follows that the quartic surface $B_L\cap H_L\subset H_L$ being reduced is an open condition, and one can check that it holds for the specific example given in \cite{HL}. 
Thus, for a very general HL-Cremona transformation,  $B_L\cap H_L$ is reduced, and the claim follows. 
\end{proof}
We will use the following fact in the next subsection.
\begin{lemm}\label{lemm:p1inplane}
        The point $p_1$ (resp., $q_1$) lies on the plane $\Lambda_L$ (resp., $\Lambda_M$) appearing in Proposition~\ref{prop:HLBLint}.
\end{lemm}
\begin{proof}
    It suffices to show $p_1\in\Lambda_L$. Recall from Proposition~\ref{prop:mult3} that $p_1$ has multiplicity 3 on $B_L.$ Up to a change of variables, we may assume that $p_1=[1:0:0:0:0]$ and $H_L+B_L$ is given by the equation
    $$
  \{x_5(x_1h_3+h_4)=0\}\subset\bP^4_{x_1,\ldots,x_5}
   $$
   where $h_k$ is a $k$-form over $\Bbbk$ in variables $x_2,x_3,x_4,x_5$ for $k=3,4$. Now assume $p_1$ is not in $\Lambda_L$, i.e., $\Lambda_L$ is given by the equation
   $$
     \{x_1+l(x_2,x_3,x_4,x_5)=0\}\subset\bP^4_{x_1,\ldots,x_5}
   $$
   for some linear form $l$. The inclusion $\Lambda_L\subset B_L\cap H_L$ implies that 
   $$
   h_4(x_2,x_3,x_4,0)=l(x_2,x_3,x_4,0)\cdot h_3(x_2,x_3,x_4,0).
   $$
  It then follows from Proposition~\ref{prop:HLBLint} that 
   $$
   \pi_L(K_1)\cup\Pi_L=\{h_3(x_2,x_3,x_4,0)=0\}\subset\bP^3_{x_2,x_3,x_4,x_5},
   $$
   which is a contradiction since $\pi_L(K_1)$ is a smooth quadric surface and not a cone. Therefore we conclude $p_1\in\Lambda_L.$
\end{proof}
   \subsection{Constructing the maps $\varphi$ and $\varphi'$} The constructions of $\varphi$ and $\varphi'$ are in symmetry. We construct $\varphi$ in detail. It is a higher-dimensional analogue of the map $(\mathrm x)$ in \cite[Section 5.1]{Du24}. Similar as above, since the point $p_1$ has multiplicity $3$ on $B_L$, we may assume $p_1=[1:0:0:0:0]$ and the equation of $D_L$ is given by 
   $$
   D_L=H_L+B_L=\{h_1(x_1h_3+h_4)=0\}\subset\bP^4_{x_1,\ldots,x_5}
   $$
   where $h_k$ is a $k$-form in variables $x_2,x_3,x_4,x_5$ not divisible by $x_2$, for $k=1,3,4$.  Now consider the map 
   \begin{align*}
          \varphi\colon (\bP^4, D_L)&\dashrightarrow(\bP^4,D'_L),\\
           (x_1:x_2:x_3:x_4:x_5)&\mapsto(x_1+\frac{h_4}{h_3}:x_2:x_3:x_4:x_5),
   \end{align*}
   where the new boundary divisor $D_L'$ has components given by equations
   \begin{align}\label{eqn:xii}
        D_L'=\Xi_L+\Xi_L'+Z_L,\quad \Xi_L=\{x_1=0\},\quad \Xi_L'=\{h_1=0\},\quad Z_L=\{h_3=0\}.
   \end{align}
   By Remark~\ref{rema:smoothcubicgeneral}, $Z_L$ is a cone over a smooth cubic surface.  Let 
   $$
\omega_{D_L}=\frac{\mathrm{d}x_1\wedge\mathrm{d}x_3\wedge\mathrm{d}x_4
  \wedge\mathrm{d}x_5}{h_1(x_1h_3+h_4)}\quad\text{and}\quad \omega_{D_L'}=\frac{\mathrm{d}x_1\wedge\mathrm{d}x_3\wedge\mathrm{d}x_4
  \wedge\mathrm{d}x_5}{x_1h_1h_3}
   $$
  be two logarithmic volume forms in the affine chart $x_2=1$ of $\bP^4_{x_1,\ldots,x_5}$ such that $\mathrm{div}(\omega_{D_L})=-D_L$ and $\mathrm{div}(\omega_{D_L'})=-D'_L$.
   One can check that $\varphi$ preserves these two forms:
   $$
   \varphi^*(\omega_{D_L'})=\frac{\mathrm{d}(x_1+h_4h_3^{-1})\wedge \dd x_3\wedge\dd x_4\wedge\dd x_5}{h_1h_3(x_1+h_4h_3^{-1})}=\frac{\dd x_1\wedge \dd x_3\wedge\dd x_4\wedge\dd x_5}{h_1(x_1h_3+h_4)}=\omega_{D_L}.
$$
It follows that $\varphi$ is a crepant birational map; its 
inverse is given by
$$
\varphi^{-1}: (x_1:x_2:x_3:x_4:x_5)\mapsto(x_1-\frac{h_4}{h_3}:x_2:x_3:x_4:x_5),
$$
and the exceptional divisors of $\varphi$ and $\varphi^{-1}$ satisfy
$$
\mathrm{Exc}(\varphi)=\mathrm{Exc}(\varphi^{-1}).
$$
It follows that 
\begin{align}\label{eqn:trivcvarphi}
    \mathbf c(\varphi^{-1})=0.
\end{align}

The following lemma in crucial in the next subsection.

\begin{lemm}\label{lemm:threelineint}
    The intersection $\Xi_L\ \cap\ \Xi_L'\ \cap\ Z_L$ is a union of three lines forming a triangle on a plane.
\end{lemm}
\begin{proof}
We keep notation from diagram \eqref{eqn:HLcremonafact} and Proposition~\ref{prop:mult3}. Let $B_L'$ be the strict transform of $B_L$ in $P$ and $H_L'$ be that of $H_L$. From equation \eqref{eqn:xii}, one sees that
    $$
    \Xi_L\cap\Xi_L'\cap Z_L\simeq E_1\cap H_L'\cap B_L'.
    $$
    In particular,  $Z_L\ \cap\ \Xi_L$ is the smooth cubic surface isomorphic to $E_1\cap B_L'$ in $P$, and $\Xi_L\cap\Xi_L'$ is isomorphic to the plane $E_1\cap H_L'$. 
    
    By Proposition~\ref{prop:HLBLint}, we know that $H_L\cap B_L$ has three components $\Pi_L\cup\Gamma_L\cup \pi_L(K_1)$, which are two planes and a smooth quadric surface. Lemma~\ref{lemm:p1inplane} implies that each component is smooth at $p_1$, and thus each of their strict transforms in $P$ intersects with $E_1$ along a line. It follows that $E_1\cap H_L'\cap B_L'$ consists of three lines in a plane.
    
Then we show that the three lines do not meet at one point. Assume they do. This implies that $\Pi_L, \Gamma_L$ and $\pi_L(K_1)$ meet along a line. By Lemma~\ref{lemm:quadint}, $\Pi_L\cap\pi_L(K_1)$ contains three points $p_1, p_2$ and $p_3$, which contradicts the generality assumption in Remark \ref{rem-points-span-plane} that $p_1, p_2$ and $p_3$ span a plane. Therefore, we conclude that $E_1\cap H_L'\cap B_L'$ consists of three lines forming a triangle, and thus the same holds for $\Xi_L\cap\Xi_L'\cap Z_L$.
\end{proof}

    

The construction of $\varphi'$ is identical to the process described above, starting from the singular point $q_1\in S_M$. By symmetry, we obtain a similar boundary $D_M'$ consisting of two hyperplanes $\Xi_M$ and $\Xi_M'$ together with a cone $Z_M$ over a smooth cubic surface. Lemma~\ref{lemm:threelineint} also holds for $\Xi_M\cap \Xi_M'\cap Z_M.$ Similarly, we have 
\begin{align}\label{eqn:trivcvarphiprime}
   \mathbf c(\varphi')=0.
\end{align}

   \subsection{Constructing the maps $\eta$ and $\eta'$} As the final step, we construct crepant birational maps $\eta$ and $\eta'$ from $(\bP^4,D_L')$ and $(\bP^4, D_M')$ to the standard toric pair $(\bP^4,\Delta_4)$ with the standard toric boundary 
   $$
   \Delta_4=\{x_1x_2x_3x_4x_5=0\}\subset\bP^4_{x_1,x_2,x_3,x_4,x_5}.
   $$
The constructions of $\eta$ and $\eta'$ are again symmetric and we focus on $\eta$. Recall from \eqref{eqn:xii} that, up to isomorphism, $D_L'$ is given by the equation 
$$
D_L'=\{x_1x_2h_3(x_2,x_3,x_4,x_5)=0\},\quad \deg(h_3)=3.
$$
Without loss of generality, we may assume $h_3$ is not divisible by $x_5$. Consider the map 
\begin{align*}
    \eta_1\colon(\bP^4,D_L')&\dashrightarrow(\bP^1\times\bP^3,\Delta'),\\
    (x_1:x_2:x_3:x_4:x_5)&\mapsto (x_1:x_5)\times(x_2:x_3:x_4:x_5),
\end{align*}
where the boundary $\Delta'$ is given by 
$$
\Delta'=\left(\{0\}\times\bP^3\right)+\left(\{\infty\}\times\bP^3\right)+\left(\bP^1\times V\right),\quad V=\{x_2h_3=0\}\subset\bP^3_{x_2,x_3,x_4,x_5}.
$$
One can check that $\eta_1$ preserves the corresponding volume forms (in the affine chart $x_5=1$) $\omega_{D'_L}$ and $\omega_{\Delta'}$ where $\mathrm{div}(\omega_{\Delta'})=-\Delta'$
$$
   \eta_1^*(\omega_{\Delta'})=\frac{\mathrm{d}x_1\wedge \dd x_2\wedge\dd x_3\wedge\dd x_4}{x_1x_2h_3}=\omega_{D'_L},
$$
and thus $\eta_1$ is a crepant birational map.
It is not hard to see that the exceptional divisors of $\eta_1$ and $\eta_1^{-1}$ are rational, and thus we have
\begin{align}\label{eqn:ceta1}
    \mathbf{c}(\eta_1)=r\cdot(\bP^3,0),\quad \text{for some integer $r\geq 0.$}
\end{align}
Note that $V\subset\bP^3$ is the union of a plane and a smooth cubic surface, where their intersection is isomorphic to 
$$
   \Xi_L\cap\Xi_L'\cap Z_L\simeq E_1\cap H_L'\cap B_L'.
$$ 
By Lemma~\ref{lemm:threelineint}, the intersection is a configuration of three lines forming a triangle on a plane. It follows from Lemma~\ref{lemm:corregcubic} that $\mathrm{coreg}(\bP^3,V)=0$. By \cite[Theorem 1.2]{Du24}, there exists a crepant birational map:
$$
\eta_2:(\bP^3,V)\dashrightarrow\left(\bP^1\times\bP^2_{y_1,y_2,y_3},(\{0\}\times\bP^2)+(\bP^1\times V')+(\{\infty\}\times\bP^2)\right), 
$$
where $V'=\{y_1y_2y_3=0\}\subset\bP^2_{y_1,y_2,y_3}$ is the union of three coordinate lines. This yields a birational map 
$$
\mathrm{id}\times \eta_2:\bP^1\times\bP^3\dashrightarrow \bP^1\times\bP^1\times\bP^2,
$$
and a crepant birational map
$$
(\mathrm{id}\times \eta_2)\circ(\eta_1)\colon (\bP^4, D_L')\dashrightarrow(\bP^1\times\bP^1\times\bP^2_{y_1,y_2,y_3},\Delta'')
$$
with the toric boundary 
\begin{multline*}
\Delta''=\left(\{0\}\times\bP^1\times\bP^2\right)+\left(\{\infty\}\times\bP^1\times\bP^2\right)+\left(\bP^1\times\{0\}\times\bP^2\right)+\\+\left(\bP^1\times\{\infty\}\times\bP^2\right)+\left(\bP^1\times\bP^1\times \{y_1y_2y_3=0\}\right). 
\end{multline*}
Notice that the exceptional divisors of the map $\mathrm{id}\times \eta_2$ and its inverse are all birational to $C'\times\bP^2$ for some curves $C',$ i.e.,
\begin{align}\label{eqn:cideta2}
    \mathbf{c}(\mathrm{id}\times\eta_2)=\sum_{C'}( C'\times\bP^2,0) \quad\text{ over finitely many curves $C'$.}
\end{align}
Now, consider the crepant birational map between toric pairs
$$
\eta_4\colon (\bP^1_{u_1,u_2}\times\bP^1_{t_1,t_2}\times\bP^2_{y_1,y_2,y_3},\Delta'')\dashrightarrow(\bP^4_{x_1,\ldots,x_5},\Delta_4),
$$
$$
(u_1:u_2)\times (t_1:t_2)\times(y_1:y_2:y_3)\mapsto (1:\frac{u_1}{u_2}:\frac{t_1}{t_2}:\frac{y_1}{y_3}:\frac{y_2}{y_3})
$$
where
$$
\quad\Delta_4=\{x_1x_2x_3x_4x_5=0\}\subset\bP^4_{x_1,\ldots,x_5}.
$$
Again, one sees that
\begin{align}\label{eqn:ceta4}
    \mathbf{c}(\eta_4)=r\cdot(\bP^3,0),\quad \text{for some integer $r\geq 0.$}
\end{align}
Lastly, we obtain a crepant birational map 
$$
\eta:=\eta_4\circ (\mathrm{id}\times \eta_2)\circ(\eta_1),\qquad \eta:(\bP^4,D_L')\dashrightarrow(\bP^4,\Delta_4).
$$
The construction above can be applied identically to $(\bP^4, D_M')$ to obtain a similar map 
$$
\eta'\colon(\bP^4,D_M')\dashrightarrow (\bP^4,\Delta_4).
$$ 

\subsection{Proof of Proposition~\ref{prop:K3} and ~\ref{prop:K3cmap}}
Put
$$
\sigma:=\eta'\circ\varphi'\circ\psi\circ\varphi^{-1}\circ\eta^{-1}.
$$ Then $\sigma$ is a crepant birational automorphism of the standard toric pair $(\bP^4,\Delta_4)$. This completes the proof of Proposition~\ref{prop:K3}. 

Now we prove Proposition~\ref{prop:K3cmap}, i.e., $\mathbf{c}(\sigma)\ne0$. By construction, we have that
$$
\mathbf{c}(\sigma)=\mathbf{c}(\eta')+\mathbf{c}(\varphi')+\mathbf{c}(\psi)+\mathbf{c}(\varphi^{-1})+\mathbf{c}(\eta^{-1}).
$$
From \eqref{eqn:trivcvarphi} and \eqref{eqn:trivcvarphiprime}, one sees that
$$
\mathbf{c}(\varphi')=\mathbf{c}(\varphi^{-1})=0.
$$
Recall that
$$
\bc(\psi)=(S_L\times\bP^1,0)-(S_M\times\bP^1,0).
$$
Consider the projection   
$$
   \mathrm{dpr}\colon \mathbf{DivBurn}_{3}(\bC)\to \bZ\cdot\left((S_L\times\bP^{1},0)-(S_M\times\bP^{1},0)\right).
$$
By construction, we have
$$
\mathbf{c}(\eta^{-1})=-\mathbf{c}(\eta_1)-\mathbf{c}(\mathrm{id}\times\eta_2)-\mathbf{c}(\eta_4).
$$
From \eqref{eqn:ceta1}, \eqref{eqn:cideta2} and \eqref{eqn:ceta4}, we see that $\bc(\eta_1)$,  $\bc(\mathrm{id}\times\eta_2)$ and $\bc(\eta_4)$ only contain classes of the form $(C'\times\bP^2,0)$ for some curve $C'$.
It follows that 
$$
(\mathrm{dpr}\circ\mathbf{c})(\eta^{-1})=(\mathrm{dpr}\circ\mathbf{c})(\eta')=0
$$
since $S_M\times\bP^1$ and $S_L\times\bP^1$ are not birational to $ C'\times\bP^2$ for any curve $C'$. Indeed, their MRC quotients are the K3 surfaces $R_M$ and $R_L$ respectively. Then we know
$$
(\mathrm{dpr}\circ\mathbf{c})(\sigma)\ne 0.
$$ 
This implies that 
$
\bc(\sigma)\ne 0.
$

\begin{coro}\label{coro:main4}
    There exists a volume preserving birational automorphism $\sigma'$ of the pair $[\bP^4,\omega_4]$, where $\omega_4$ is the standard torus invariant volume form, such that 
    $$
   0\ne\mathbf{c}(\sigma')\in\Burnf_3(\bC).
    $$
    In particular, $\Bir_{\bC}(\bP^n,\omega_n)$ is not generated by pseudo-regularizable elements for $n\geq 4$.
\end{coro}
\begin{proof}
    Let $\sigma$ be the crepant birational automorphism of $(\bP^4,\Delta_4)$ constructed in Proposition~\ref{prop:K3}. It follows from Corollary~\ref{coro:volumepreserve} that there exists a positive integer $N$ such that $\sigma^N$ is a volume preserving birational automorphism of $[\bP^4,\omega_4]$. Put $\sigma'=\sigma^N$. We have 
    $$
    0\ne c(\sigma')=N\cdot c(\sigma)\in\Burn_{3}(\bC)
    $$
    where $c$ is the Lin-Shinder invariant. 
    By Remark~\ref{rema:samec}, we know that
    $$
    0\ne\bc(\sigma')\in\Burnf_3(\bC).
    $$
   Similar to the proof of Corollary \ref{coro:nonsimple3}, the map $\sigma'$ extends to a volume preserving birational automorphism $\tilde \sigma$ of the pair  
$
[\bP^{4+r},\omega_{4+r}]
$
for $r\geq 0$ with 
$$
    0\ne\bc(\tilde\sigma)\in\Burnf_{r+3}(\bC).
    $$
        The last assertion then follows from Example \ref{exam:vanishpseudo}.
\end{proof}

\begin{coro}\label{coro:nonsimple4}
     There exists a surjective homomorphism when $n\geq4:$
     $$
     \Bir_{\bC}(\bP^n,\omega_n)\to A
     $$
 where $A=\bigoplus_J \bZ$ and $J$ is a set of the same cardinality as $\bC.$ In particular, $\Bir_{\bC}(\bP^n,\omega_n)$ is not simple when $n\geq 4$.
\end{coro}
\begin{proof}
  
   Let $I$ be the set of isomorphism classes of K3 surfaces over $\bC$ of degree 12 and Picard rank 1 giving rise to a diagram~\eqref{eqn:HLcremonafact}. Proposition~\ref{prop:K3} shows that an HL-Cremona transformation associated with a very general such $K3$ surface gives rise to such a diagram. It follows that the set $I$ has the cardinality of $\bC$ \cite[Lemma 3.8]{SL}. Let $J$ be the set of unordered pairs 
    $$
    J:=\{(S,S'): S\in I\}
    $$
where $S'$ is the unique Fourier-Mukai partner of $S$ which is not isomorphic to $S.$
    Then $J$ also has the cardinality of $\bC$. For $n\geq 4$, put 
    $$
    A':=\bigoplus_{J}\bZ\cdot([S\times\bP^{n-3},0]-[S'\times\bP^{n-3},0])\subset\Burnf_{n-1}(\bC).
    $$
    Since distinct elements in $I$ are not stably birational to each other, $A'\simeq \bigoplus_J\bZ$. Consider the projection
    $$
    \mathrm{pr}:\Burnf_{n-1}(\bC)\to A'.
    $$
    Then the image of the map
    $$
    \mathrm{pr}\circ \bc: \Bir_{\bC}(\bP^n,\omega_n)\to A'
    $$
    is a free abelian group $A$ also isomorphic to $\bigoplus_J\bZ.$

\end{proof}

\bibliographystyle{alpha}

\bibliography{CY}

\end{document}